\documentclass[12pt,oneside]{amsart}
\usepackage[latin1]{inputenc}
\usepackage{amsmath,a4,amssymb,amscd}
\usepackage{typearea,lmodern}
\usepackage{url}

\input diagxy

\newcounter{lemmacounter}

\newcounter{propcounter}

\newcounter{tmpcounter}
\newcounter{thmcounter}

\numberwithin{equation}{section}
\numberwithin{excounter}{section}
\numberwithin{defcounter}{section}
\numberwithin{thmcounter}{section}
\numberwithin{corcounter}{section}
\numberwithin{lemmacounter}{section}
\numberwithin{propcounter}{section}

\newtheorem{theorem}[thmcounter]{Theorem}
\newtheorem{proposition}[propcounter]{Proposition}
\newtheorem{lemma}[lemmacounter]{Lemma}

\newtheorem{lemma*}[tmpcounter]{Lemma}

\newtheorem*{conjecture*}{Conjecture}
\newtheorem{proposition*}[tmpcounter]{Proposition}

\newcommand{\IH}{\mathbf H}
\newcommand{\IN}{\mathbf N}

\newcommand{\IP}{\mathbf P}
\newcommand{\IC}{\mathbf C}
\newcommand{\IZ}{\mathbf Z}
\newcommand{\IQ}{\mathbf Q}
\newcommand{\IR}{\mathbf R}
\newcommand{\IQbar}{\overline{\mathbf Q}}

\newcommand{\hpS}[1]{H_{#1}}
\newcommand{\hp}[2]{\hpS{#2}({#1})}

\newcommand{\heightS}{h}
\newcommand{\ntheightS}{\hat h}
\newcommand{\ntheight}[1]{\ntheightS{(#1)}}
\newcommand{\height}[1]{\heightS({#1})}
\newcommand{\heightlbS}[1]{h_{#1}}
\newcommand{\heightlb}[2]{\heightlbS{#2}({#1})}
\newcommand{\ntheightlbS}[1]{{\hat h}_{#1}}
\newcommand{\ntheightlb}[2]{\ntheightlbS{#2}({#1})}

\newcommand{\Fh}[1]{h_F({#1})}

\newcommand{\M}[2]{{\rm Mat}_{#1}({#2})}
\newcommand{\SSL}[1]{{\rm SL}_{#1}}
\newcommand{\SL}[2]{\SSL{#1}({#2})}

\newcommand{\remtor}[1]{{#1}^{\star}}

\newcommand{\ns}[1]{{#1}^{\rm ns}}

\newcommand{\dom}[1]{{\rm dom}({#1})}

\renewcommand{\O}[1]{\mathcal{O}({#1})}

\newcommand{\atopx}[2]{\genfrac{}{}{0pt}{}{#1}{#2}}

\newcommand{\trans}[1]{{#1}^{{\tt T}}}

\newcommand{\T}{T}
\newcommand{\U}{\Lambda}

\newcommand{\B}{{Y(2)}}
\newcommand{\EL}{{\mathcal{E}_L}}
\newcommand{\AL}{{\mathcal{A}_L}}
\newcommand{\piL}{{\pi_L}}

\newcommand{\spec}[1]{{\rm Spec\,}{#1}}

\newcommand{\ssm}{\smallsetminus}

\newcommand{\citeEGAIVII}{{\cite[EGA ${\rm IV}_{\rm 2}$]{EGAIV}} }
\newcommand{\citeEGAIVIII}{{\cite[EGA ${\rm IV}_{\rm 3}$]{EGAIV}} }
\newcommand{\citeEGAIVIV}{{\cite[EGA ${\rm IV}_{\rm 4}$]{EGAIV}} }

\begin{document}
\title{Special Points on Fibered Powers of Elliptic Surfaces}
\author[Philipp Habegger]{P. Habegger}

\maketitle

\begin{abstract}
Consider a fibered power of an elliptic surface.
We  characterize its  subvarieties 
 that contain a Zariski dense set of points that are torsion
points in fibers with complex multiplication. This result can be
viewed as a mix of the Manin-Mumford and Andr\'e-Oort Conjecture and
is related to a conjecture of Pink \cite{Pink05}.
The main technical tool is a new height inequality. We  also use
it to give another proof of a  case of Gubler's result on
the Bogomolov Conjecture over function fields \cite{Gubler:Bogo}.
\end{abstract}

\section{Introduction}

In this paper we verify a combination of the Manin-Mumford and 
 Andr{\'e}-Oort Conjecture for a class of abelian schemes:
  fibered powers of an elliptic surface.
The latter conjecture can also be combined with the Mordell-Lang
Conjecture and we obtain results in this context.
A common generalization for all three conjectures was
proposed by Pink \cite{Pink05,Pink}.
An important tool in our proofs is a new height inequality on
subvarieties of the ambient abelian scheme. 
This may be of independent interest as it
generalizes to higher dimension a height theoretic result of Silverman
used in the proof of his
Specialization Theorem \cite{Silverman}. In a third application we
use our height
inequality  to recover the Bogomolov Conjecture for
products of elliptic curves over the function field of a curve. 

Before stating the results we introduce the relevant class of abelian
schemes.
Let $S$ be an irreducible and non-singular quasi-projective curve defined over
$\IQbar$, the algebraic closure of $\IQ$ in $\IC$.
Let $\mathcal{E}\rightarrow S$ be an abelian scheme over $S$ whose
fibers are elliptic curves. For an integer $g\ge 1$ we let 
$\mathcal{A}$ denote the $g$-fold fibered power
$\mathcal{E}\times_S\cdots\times_S\mathcal{E}$. This is also an
abelian scheme over $S$. Let  $\pi$ be the
structural morphism $\mathcal{A}\rightarrow {S}$. If $s\in
S(\IC)$, it is convenient to write $\mathcal{A}_s$ for $\pi^{-1}(s)$,
which is the $g$-th power of an elliptic curve. 
For reasons explained below, our results require $S$, and so
$\mathcal{A}$, to be defined over $\IQbar$. Nevertheless,
we will speak of subvarieties of $\mathcal{A}$ defined over $\IC$ by
extending scalars without further mention.

Our main interest lies in the case where
there are sufficiently many non-isomorphic
abelian varieties among the fibers of $\mathcal{A}\rightarrow S$.
We call $\mathcal{E}$ (or $\mathcal{A}$)
  isotrivial if $\mathcal{E}\rightarrow S$ becomes a constant family
  after a finite \'etale base change.

We now introduce the special points and special subvarieties of $\mathcal{A}$.
 We call a point in $\mathcal{A}(\IC)$ special if it
is a torsion point of its respective fiber and if this fiber 
 has complex multiplication. 
An irreducible closed subvariety of $\mathcal{A}$ 
 is called special 
\begin{enumerate}
\item [(i)]if it is an irreducible component of an algebraic subgroup of  
 $ \mathcal{A}_s$ with $s \in S(\IC)$ such that
$\mathcal{A}_s$ has complex
multiplication,
\item[(ii)] or if it is an irreducible
component of a flat subgroup scheme of $\mathcal{A}$; cf.
Section \ref{sec:sbgrpscheme} for the definition of flat subgroup schemes.
\end{enumerate}
The point of a zero-dimensional special subvariety
is a special point.

An explicit and important example of an abelian scheme is  the
Legendre family of elliptic curves over the modular curve $\B = \IP^1\ssm
\{0,1,\infty\}$ taken as defined over $\IQbar$.
Indeed, the affine equation
\begin{equation*}
  y^2  = x(x-1)(x-\lambda)
\end{equation*}
determines a subvariety of $\IP^2\times \B$ which we denote with 
$\EL$. We let $\piL$ denote the morphism  which
projects $\EL$ to $\B$. Then $\EL$ is an abelian scheme over $\B$ and the
fibers of $\piL$ are elliptic curves, cf. Section \ref{sec:heights}. Any elliptic curve over
$\IQbar$ has a Legendre model, 
so it is isomorphic to some fiber of $\piL$.
This shows that $\EL$ is not
isotrivial. We write $\AL$ for the $g$-fold fibered power of
$\EL$. 

From a different point of view, $\EL$ and $\AL$ 
can be realized as
connected mixed Shimura varieties, cf Pink's Construction 2.9
\cite{Pink05}.
This additional structure comes with a natural notion of special
points  which coincides with our notion by 
Pink's Remark 4.13.

In an abelian variety, the Manin-Mumford Conjecture characterizes 
irreducible components of algebraic subgroups
as those irreducible subvarieties that 
 contain a Zariski dense set of torsion
points. Its first proof is due to Raynaud \cite{Raynaud:MM}.
The Andr\'e-Oort Conjecture, on the other hand,
expects special subvarieties of Shimura varieties 
to be precisely those irreducible subvarieties that contain
 a Zariski dense set of special points. 
 Klingler and Yafaev have announced a proof \cite{KlinglerYafaev} which assumes the 
Generalized Riemann Hypothesis.

Our first result characterizes subvarieties of $\mathcal{A}$
containing a Zariski dense set of special points. 

\begin{theorem}
\label{thm:special1}
Let $\mathcal{A}$ be as above and let us assume that $\mathcal{A}$ is
not isotrivial.
An irreducible closed subvariety of $\mathcal{A}$ defined over $\IC$ 
contains a Zariski dense set of
  special points if and only if it is special. 
\end{theorem}

It follows that the Zariski closure of a set of special points in
$\mathcal{A}$ is a finite union of special subvarieties.

This theorem generalizes a result
of  Andr{\'e} \cite[Lecture IV]{Andre01} which holds for curves with $\mathcal{A}$
the Legendre family of elliptic curves (so $g=1$).
Later, Pila  \cite{PilaCurves} gave a  proof of
 Andr\'e's  statement using a different approach. 

The assumption that $\mathcal{A}$ is  not isotrivial is necessary. 
We construct a
counterexample for the constant abelian scheme  $E\times\IP^1$
where $E$ is an elliptic curve with complex multiplication.
The special subvarieties as in (i) above are
of the form $E\times \{s\}$ for some $s\in\IP^1(\IC)$;
those as in (ii) are $\{P\}\times\IP^1$ with $P$ a torsion point of $E$.
Let $C$ be a curve in $E\times \IP^1$ that is not equal to
$\{P\}\times\IP^1$ for any $P\in E(\IC)$ and not of the form $E\times \{s\}$. On considering the projection of $C$ to
$E$ we find that our curve contains infinitely many points $(P,s)$ with $P$
torsion. All these points are  special, but $C$ is not.


The proof of Theorem \ref{thm:special1} relies on a height
inequality to be described in more detail below. 
Another ingredient is a finiteness statement of Poonen \cite{Poonen:MRL01}
on elliptic curves with complex multiplication and bounded Faltings
height. His proof relies on results of Colmez and Nakkajima-Taguchi.
We cannot work with more general abelian schemes 
 because these results are confined to elliptic curves for the
moment.


Following a suggestion of Zannier, we investigate a second ``special
topology'' on $\mathcal{A}$ relative to a fixed elliptic curve $E$ defined over $\IC$. 
A point in $\mathcal{A}(\IC)$ is called $E$-special 
if it is a torsion point in its respective fiber  and if this fiber is
isogenous to $E^g$.
Special subvarieties of
$\mathcal{A}$ are defined in a similar fashion as above. Explicitly, 
an
irreducible closed subvariety of $\mathcal{A}$ defined over $\IC$ is called $E$-special 
\begin{enumerate}
\item [(i)]
if it is an irreducible component of an algebraic subgroup of
 $ \mathcal{A}_s$ with $s\in S(\IC)$ such that
$\mathcal{E}_s$ is isogenous to $E$,
\item[(ii)] or if it is an irreducible
component of a flat subgroup scheme of $\mathcal{A}$. 
\end{enumerate}
The point of an zero-dimensional  $E$-special subvariety is
$E$-special.

The set of  $E$-special points of $\AL$ is a 
Hecke orbit, as defined in Section 3 \cite{Pink05}, of the zero element of an appropriate fiber, cf. the proof of Proposition 5.1 there. 


\begin{theorem}
\label{thm:specialE}
Let $\mathcal{A}$ be as above and let us assume that $\mathcal{A}$ is
not isotrivial.
Let $E$ be an elliptic curve defined over $\IQbar$.
 An irreducible closed subvariety of $\mathcal{A}$ defined over $\IC$
 contains a Zariski dense set of
  $E$-special points if and only if it is $E$-special.  
\end{theorem}

So,  the Zariski closure of a set of $E$-special points in
$\mathcal{A}$ is a finite union of $E$-special subvarieties.


In addition to the height inequality which is also used in Theorem
\ref{thm:special1}, the theorem above relies
on a result of Szpiro and Ullmo \cite{SU:variation}. They describe the
distribution of the Faltings height in a fixed isogeny class of
elliptic curves without complex multiplication. 
As was the case with the finiteness statement
 of Poonen, 
a version of this result for more general abelian varieties 
would be needed to treat abelian schemes with more general fibers.


The main  technical tool in the proofs of Theorems
\ref{thm:special1} and \ref{thm:specialE} is a height inequality
on subvarieties of $\mathcal{A}$
given in Theorem \ref{thm:main}. 
It relates the restrictions of two
different height functions on $\mathcal{A}$
to a fixed subvariety.
Since our heights are only defined when dealing with algebraic points
we shall assume that $S$ and $E$ are defined over $\IQbar$.

The first height function is derived from a height on $S$; it
measures the corresponding fiber in $\mathcal{A}$.
 We may assume that $S$
is a Zariski open subset of an irreducible and non-singular projective curve $\overline
S$ over $\IQbar$. 
On $\overline S$ we fix a line bundle $\mathcal{L}$.
Given this pair we may choose a height function
$\heightlbS{\overline S,\mathcal{L}}:\overline S(\IQbar)\rightarrow
\IR$,
 cf. Section \ref{sec:heights}. 
The first height is  just the composition $\heightlbS{\overline
  S,\mathcal{L}}\circ \pi : \mathcal{A}(\IQbar)\rightarrow \IR$.

The second height function
 $\ntheightlbS{\mathcal{A}}:\mathcal{A}(\IQbar)\rightarrow [0,\infty)$
  is  more closely related to the group structure on the fibers 
of $\mathcal{A}\rightarrow S$ above points in $S(\IQbar)$.
Indeed, for any $s\in S(\IQbar)$
the fiber $\mathcal{A}_s$ is the $g$-th power of an elliptic curve.
It is equipped with the so-called N\'eron-Tate height
$\mathcal{A}_s(\IQbar)\rightarrow [0,\infty)$ which we describe more
  thoroughly in Section \ref{sec:heights}.
Letting $s$ vary over $S(\IQbar)$ we obtain a N\'eron-Tate height
$\ntheightlbS{\mathcal{A}}:\mathcal{A}(\IQbar)\rightarrow[0,\infty)$.


Our two height functions are  unrelated in the following sense.
It is not difficult to construct
an infinite sequence of points  $P_1,P_2,\ldots\in \mathcal{A}(\IQbar)$ such that
$\ntheightlb{P_k}{\mathcal{A}}$ is constant  and
$\heightlb{\pi(P_k)}{\overline S,\mathcal{L}}$ unbounded. For example, 
it suffices to take  $P_k$ any torsion point in $\mathcal{A}_{\pi(P)}$
 and the sequence $\pi(P_k)$ of unbounded height. Then
 $\ntheightlb{P_k}{\mathcal{A}}=0$
since the N\'eron-Tate height vanishes on torsion points.

If $\mathcal{A}$ is not isotrivial, the situation
changes when our points lie on an
 irreducible subvariety $X\subset \mathcal{A}$ which is
 not  ``special'' in a slightly weaker sense than above.
This is the content of the height inequality in Theorem \ref{thm:main}.
We will bound
$\heightlbS{\overline S,\mathcal L}\circ \pi$ from above linearly 
in terms of $\ntheightlbS{\mathcal{A}}$ when restricted to a certain 
natural Zariski open and non-empty subset of $X$. 
We define this subset now. 

For an irreducible closed subvariety
$X\subset\mathcal{A}$ defined over $\IC$ we set
\begin{equation*}
  \remtor{X} = X \ssm \bigcup_Z Z
\end{equation*}
where $Z$ runs over all
closed subvarieties of $X$ that are
 irreducible components
of  flat  subgroup schemes of $\mathcal{A}$.
The fact that $\remtor{X}$ is Zariski open is
not immediately obvious since the union may be infinite.
But it is part of the theorem below and 
 will follow from the
 Manin-Mumford Conjecture applied to the
 generic fiber of $\mathcal{A}\rightarrow S$. 
We will also obtain 
 a necessary and sufficient condition for
the non-emptiness of $\remtor{X}$. 


\begin{theorem}
\label{thm:main}
Let $\mathcal{A}$ be as above and 
  let $X\subset\mathcal{A}$ be an irreducible closed subvariety
  defined over $\IC$. 
  \begin{enumerate}
\item[(i)] The set $\remtor{X}$ is Zariski open in $X$. 
  It is empty if and only if $X$ is  an irreducible component
  of a flat subgroup scheme of $\mathcal{A}$. 
  \item [(ii)] If $X$ is defined over $\IQbar$ and if $\mathcal{A}$ is
    not isotrivial there exists a constant $c > 0$ such that
  \begin{equation}
\label{eq:mainthmineq}
    \heightlb{\pi(P)}{\overline S,\mathcal{L}} \le c \max\{1,\ntheightlb{P}{\mathcal{A}}\} 
\quad\text{for all}\quad P\in \remtor{X}(\IQbar). 
  \end{equation}
  \end{enumerate}
\end{theorem}

If $X$ is a curve, then Theorem \ref{thm:main}(ii)
can be proved using  Silverman's  Theorem B
\cite{Silverman}; we state it as Theorem \ref{thm:silverman} below. 
In fact, Silverman's result provides a more precise estimate for more
general abelian schemes. 
His proof depends on the fact that an irreducible 
projective curve has infinite cyclic N\'eron-Severi group.
The advantage of our theorem is that it can
handle subvarieties of arbitrary dimension. 


 Masser and Zannier \cite{MZ:torsionanomalous}
proved that a certain explicit curve in
$\EL\times_{\B} \EL$
contains only finitely many points which are torsion in their
respective fibers. Their result is also related to Pink's general
conjecture
\cite{Pink}.
One step in their argument required a height bound as in Theorem \ref{thm:main} 
for curves. For this they used Silverman's result 
mentioned further up. One could hope that our height theoretic result
may  play a role in a generalization of Masser and Zannier's result
to higher dimensional subvarieties. 

The particular abelian scheme $\AL\rightarrow \B$ 
defined using the Legendre family
plays a central role in the proof of
Theorem \ref{thm:main}. 
By adding level structure to
$\mathcal{E}\rightarrow S$ we will  be able to reduce the
proof to the case  $\mathcal{A}=\AL$ and $S=\B$.
This allows us to exploit the very explicit nature of the Legendre family.

We briefly sketch the lines of the  proof in this setting.
Perhaps surprisingly, the basic strategy is to 
 construct sufficiently 
many points on $X$ which are torsion in their respective fibers.
This is done in Proposition
\ref{prop:counting} if $X$ is a hypersurface. 
The existence of many such torsion points has implications for a certain
intersection number on an appropriate compactification of $X$.
This information will be  used to establish the existence of an
auxiliary non-zero global section of a certain line bundle. 
By arguments from height theory this
global section is ultimately responsible for the inequality in
Theorem \ref{thm:main}(ii). 
If $X$ is not a hypersurface, then we will   apply an inductive argument.

There is an implicit restriction on $X$ in part (ii) of the theorem
above. Namely, $\remtor{X}\not=\emptyset$ since the statement is
trivial otherwise.
This suggests that there must be an obstruction  in the  sketch above.
Indeed, it is the argument in Proposition \ref{prop:counting}
 which may fail if $X$ is an irreducible component of a flat subgroup
 scheme.  
Part of the proof of this  proposition 
concerns the Zariski denseness of what one might call an analytic
subgroup scheme of $\AL$. This is done by studying the
 local monodromy of our abelian scheme
 around the cusps $0$ and $1$ of $\B$. 
  Roughly speaking,  monodromy
allows us 
 to extract information from the hypothesis
 $\remtor{X}\not=\emptyset$.
For abelian schemes, local
 monodromy is known  to be quasi-unipotent. For our specific
 $\AL$  we will see that it is even unipotent around $0$ and $1$. We
 will  show that the nilpotent part has sufficiently large rank.
This allows us to 
 apply Kronecker's Theorem from diophantine approximation giving the
 argument an ergodic flavor. 

Our claim on Zariski denseness  can be rephrased by saying that a certain
set of functions is algebraically independent over the field
$\IC(\lambda)$. These functions turn out closely related to
 solutions of type VI Painlev\'e
differential equations. In the setting we consider,
 they are known to be
transcendental over $\IC(\lambda)$. But  algebraic independence
 seems to be new. 


Local monodromy of an abelian scheme over a projective base
is of finite order. Hence, the nilpotent part is trivial.
It would be interesting to see if and how our approach can
  adapt to abelian schemes over  curves lacking cusps.





We come to a final application of the height inequality.
The Bogomolov Conjecture for abelian varieties defined over a number field
 generalizes the
Manin-Mumford Conjecture. 
  Whereas the Manin-Mumford Conjecture describes the
  distribution of torsion points on subvarieties of abelian varieties,
  the Bogomolov Conjecture governs   those points which merely have
  small N\'eron-Tate height. 

Over number fields, the Bogomolov Conjecture is a theorem due to the
work of Ullmo and Zhang. 
It has an analog for abelian varieties  defined over
 function fields since one can also define the N\'eron-Tate height in
 this setting. The Bogomolov Conjecture is open in the
 context of function fields.
But Gubler  \cite{Gubler:Bogo}  has made  important progress by proving
it if the abelian variety is totally degenerate at one place
of the function field.

We prove  the Bogomolov Conjecture for
the power of an elliptic curve defined over the function field of a
curve and with non-constant $j$-invariant.
This abelian variety can be realized as  
 the generic fiber of some  $\mathcal{A}\rightarrow S$ as
in Theorem \ref{thm:main}.  Here our height inequality comes into the
picture;  we shall combine it with the more precise
statement of Silverman which holds for curves.
In fact this particular case of the Bogomolov Conjecture is covered by Gubler's work. But our
approach differs from his and provides another approach to this problem.

Let $K$ be the function field of an irreducible non-singular projective curve
defined over $\IQbar$ and let $E$ be an elliptic curve defined over
$K$.
The $j$-invariant of $E$ is an element of $K$; we 
call it non-constant if it lies in $K\ssm \IQbar$.
 Let $\overline K$ be
an algebraic closure of $K$.
We  fix an ample and symmetric line bundle
on $E$. This induces a
N\'eron-Tate height function $\ntheightS : E^g(\overline K)\rightarrow
[0,\infty)$, see Section \ref{sec:bogo} for references in the function
  field setting.


\begin{theorem}
\label{thm:bogo}
Let $K,\overline K,E,$ and $\ntheightS$ be as above. 
We shall assume that the $j$-invariant of $E$ is non-constant. 
  Let $X\subset E^g$ be an irreducible closed subvariety defined over
  $\overline K$
 which is not an irreducible component of an algebraic
  subgroup of $E^g$. 
There exist $\epsilon > 0$ and a Zariski closed
  proper subset $Z\subset X$ such that $P \in (X\ssm
  Z)(\overline K)$ implies $\ntheight{P} \ge \epsilon$. 
\end{theorem}

The article is organized as follows.
In Section \ref{sec:notation} we introduce much of the notation used
throughout later sections. Section \ref{sec:torsionpts} contains the
estimate on counting torsion points alluded to in the  
sketch above.
We will use it in Section \ref{sec:internumb} to prove a
preliminary height inequality. In Section
\ref{sec:mainresults} we then deduce Theorem
\ref{thm:main}. 
In the same section we also prove Theorems \ref{thm:special1} and
\ref{thm:specialE}.
Finally, Theorem \ref{thm:bogo} is shown in Section \ref{sec:bogo}.

The basis of this work was laid at the Centro di
Giorgi in Pisa in June 2009.
It is my pleasure to thank Umberto Zannier for inviting me and for
the many fruitful discussion we had. 
I am grateful to Daniel Bertrand for his comments on transcendence properties and Painlev\'e
equations
and mentioning Zarhin and Manin's
paper.
I  also thank
Walter Gubler and  Richard Pink for answering
  several  questions. The author was supported
by an ETH Fellowship grant and by the Scuola Normale Superiore in Pisa.

\section{Preliminaries}
\label{sec:notation}

\subsection{One Abelian Scheme with Two Heights}
\label{sec:heights}

Let $\B=\IP^1\ssm \{0,1,\infty\}$ and 
let $\EL$ be the closed subvariety of  $\IP^2\times \B$ given by
\begin{equation*}
  \left\{([x:y:z],[\lambda:1]);\,\, zy^2 = x(x-z)(x-z\lambda)\right\}
\subset \IP^2\times \B.
\end{equation*}
We let $\piL=\EL \rightarrow \B$ denote the projection onto
the second factor. 
Each fiber of $\piL$ is an elliptic curve given in Legendre form.
The zero section
 $\epsilon_{L} : \B\rightarrow \EL$
is defined as $\epsilon_{L}(\lambda) =
 ([0:1:0],\lambda)$; it is a closed immersion.
The addition morphism on each fiber of $\EL\rightarrow\B$ extends to 
 a
 morphism $\EL\times_\B \EL\rightarrow \EL$. Similarly, we have
 a morphism $\EL\rightarrow\EL$ which is fiberwise the
 inversion. Therefore $\mathcal{E}$ is 
a group scheme over $\B$.
The morphism $\piL$ is proper
because it is a composition of the closed immersion
$\EL\hookrightarrow \IP^2\times \B$ and the projection morphism
$\IP^2\times \B\rightarrow \B$ which is proper.
The morphism  $\piL$ is smooth and geometrically connected
because its fibers are elliptic curves.  Hence
$\EL$ is an abelian scheme over $\B$. 

Throughout this paper $g\ge 1$ is an integer and 
$\AL = \EL\times_\B \cdots \times_\B \EL$ is
 the  $g$-fold  fibered  power of $\EL$ over $\B$. 
It is  an abelian scheme
 over $\B$ of dimension $g+1$.
We have a natural embedding $\AL \subset (\IP^2)^g\times \B$. 
By abuse of notation we write $\piL$ for the canonical map $\AL\rightarrow
\B$ and
$\epsilon_L:\B\rightarrow\AL$  for the zero section.

We consider the absolute logarithmic Weil height 
 $\heightS : \IP^n(\IQbar) \rightarrow [0,\infty)$ and sometimes call
  it the projective height. For a definition
  and some basic properties we refer to Chapter 1.5 \cite{BG}. Since we
  have a natural inclusion $\B\subset \IP^1$, the projective height
restricts to a height  $\heightS:\B(\IQbar)\rightarrow [0,\infty)$. 
The Weil height of an algebraic number $x$ is the projective height of
$[x:1]\in\IP^1(\IQbar)$.

Say $Z$ is a projective variety defined over $\IQbar$ and
$\mathcal{L}$ a line bundle 
on $Z$.  
This pair determines an 
 equivalence class $\heightlbS{Z,\mathcal{L}}$
of real valued functions $Z(\IQbar)\rightarrow\IR$
where two functions are taken to be equivalent if the absolute value
of their difference is uniformly bounded from above. 
The association $(Z,\mathcal{L})\mapsto \heightlbS{Z,\mathcal{L}}$
 has useful functorial properties which we use freely
throughout this paper. For more information on
these and a 
 construction  we refer to Chapter 2 of Bombieri and
Gubler's book \cite{BG}. 
It is sometimes convenient to 
 use the same symbol $\heightlbS{Z,\mathcal{L}}$
for a specific representative in the equivalence class. 
We will point out such a choice.

We discuss two  notions for the height of a point
 $P=(P_1,\ldots,P_g,\piL(P))\in \AL(\IQbar)$. 
As in the introduction we could use $\height{\piL(P)}$ to gauge the fiber
containing $P$. It is sometimes more convenient 
to work with ``total height''  given by
\begin{equation}
\label{eq:deftotalheight}
  \heightlb{P}{\AL} = 
\height{P_1}+\cdots+\height{P_g}+\height{\piL(P)},
\end{equation}
we recall $P_1,\ldots,P_g\in\IP^2(\IQbar)$.

Let $E$ be any elliptic curve defined over $\IQbar$.
The zero element of $E$ considered as a Weil divisor determines a line bundle
$\mathcal{L}$ on $E$. 
There is a rational function $x$ on $E$
whose only pole is at the zero element and of order two there.
Then $x$ extends to a morphism $E\rightarrow \IP^1$ and
a valid choice of representative for $\heightlbS{E,\mathcal{L}}$ 
is $\frac 12 \heightS \circ x : E(\IQbar)\rightarrow [0,\infty)$
with $\heightS$ the projective height.
Tate's Limit Argument, cf. Chapter 9.2 \cite{BG},
enables us to choose a canonical element in the equivalence class
 $\heightlbS{E,\mathcal{L}}$.
We let $\ntheightlbS{E}:E(\IQbar)\rightarrow [0,\infty)$ denote this
  element and call it  the
  N\'eron-Tate height. 
The N\'eron-Tate height has the advantage that if $E'$ is an elliptic curve over $\IQbar$ and
$f:E\rightarrow E'$ is an isomorphism of elliptic curves, then functorial properties
of the height imply
$\ntheightlb{f(P)}{E'} = \ntheightlb{P}{E}$ for all $P\in E(\IQbar)$.

Let $\mathcal{A}$ and $S$  be as in the introduction. If $s\in
S(\IQbar)$ and $P=(P_1,\ldots,P_g)\in \mathcal{E}_s^g(\IQbar)$, we set
\begin{equation}
\label{eq:defntheight}
\ntheightlb{P}{\mathcal{A}} = \ntheightlb{P_1}{\mathcal{E}_s}+\cdots
+\ntheightlb{P_g}{\mathcal{E}_s}\ge 0
\end{equation}
 and  call this the N\'eron-Tate height on $\mathcal{A}$.
Of course, this also determines a N\'eron-Tate height on $\AL$.

\subsection{Period Map}

For  $\tau\in \IH$, where $\IH\subset\IC$ is the upper half-plane,  we have
the Weierstrass  function
\begin{equation*}
 \wp(z;\tau):\IC\ssm (\IZ+\tau \IZ)\rightarrow \IC,
\end{equation*}
which is holomorphic on its domain and $\IZ+\tau\IZ$-periodic; a
reference is Chapter 1 \cite{Lang:elliptic}.
If $\tau,\tau'\in\IH$ generate the same lattice, i.e.
$\IZ+\tau\IZ=\IZ+\tau'\IZ$, then  $\wp(\cdot;\tau)=\wp(\cdot;\tau')$. 
We recall the
 classical equalities
\begin{equation}
\label{eq:transfwp}
  \wp(\alpha z;\alpha \tau) = \alpha^{-2} \wp(z;\tau),
\quad\text{and}\quad
  \wp'(\alpha z;\alpha \tau) = \alpha^{-3} \wp(z;\tau)
\end{equation}
which hold if the corresponding expressions are well-defined.
The Weierstrass function and its derivative satisfy the differential
equation
\begin{equation*}
  \wp'(z;\tau)^2 = 4(\wp(z;\tau)-e_1(\tau))(\wp(z;\tau)-e_2(\tau))(\wp(z;\tau)-e_3(\tau))
\end{equation*}
where
\begin{equation*}
  e_1(\tau)= \wp(\tau/2;\tau),\quad 
  e_2(\tau)= \wp(1/2;\tau),\quad\text{and}\quad
  e_3(\tau)= \wp((1+\tau)/2;\tau)
\end{equation*}
are pairwise distinct complex numbers for fixed $\tau$. 
Thus
\begin{equation*}
  \U(\tau) = \frac{e_3(\tau)-e_1(\tau)}{e_2(\tau)-e_1(\tau)},
\end{equation*}
 is a well-defined holomorphic map $\IH\rightarrow
\IC\ssm\{0,1\}=\B(\IC)$.

We  now exhibit a local inverse for $\U$.
Gauss's hypergeometric function 
\begin{equation*}
F(\lambda)= {}_2 F_1\left(\frac 12,\frac 12, 1,\lambda\right) 
= \sum_{n=0}^{\infty} \frac{(2n)!^2}{ 2^{4n} n!^4} \lambda^n
\end{equation*}
is holomorphic on the open unit disc in $\IC$. 
We set
\begin{equation*}
  \omega_1(\lambda) =F(\lambda)\pi\quad\text{and}\quad
\omega_2(\lambda) =  F(1-\lambda)\pi i
\end{equation*}
and obtain two  functions, both holomorphic on
\begin{equation*}
 \Sigma = \{\lambda\in\IC;\,\, |\lambda|<1\text{ and }|1-\lambda|<1\}.
\end{equation*}
By Theorem 6.1, page 184 \cite{Husemoeller} the complex numbers
$\omega_1(\lambda),\omega_2(\lambda)$ are periods of an elliptic curve
for $\lambda\in\Sigma$. So they are $\IR$-linearly independent
and in particular, $\omega_1(\lambda)\not=0$.
We obtain a holomorphic map $T:\Sigma\rightarrow\IC$ defined by
\begin{equation}
\label{def:tau}
 \T(\lambda)= \frac{ \omega_2(\lambda)}{\omega_1(\lambda)}
 = \frac{F(1-\lambda)}{F(\lambda) }i.
\end{equation}
Since $\T(1/2) = i$ lies in $\IH$ we have
 $\T(\Sigma)\subset \IH$. 

\begin{lemma}
\label{lem:taulocalinv}
\begin{enumerate}
\item [(i)]
 For any $\lambda\in\Sigma$ we have
$\U(\T(\lambda)) = \lambda$.   
\item [(ii)] Let $\tau\in\IH$, we have
  \begin{equation}
\label{eq:modularbusiness1}
e_k(\tau+2)=e_k(\tau)\quad\text{and}\quad
e_k\left(\frac{\tau}{-2\tau + 1}\right)
= (-2\tau+1)^2 e_k(\tau)
\quad\text{for}\quad 1\le k\le 3
  \end{equation}
as well as
  \begin{equation}
\label{eq:modularbusiness2}
    \U(\tau + 2) = \U(\tau)\quad\text{and}\quad
    \U\left(\frac{\tau}{- 2\tau+1}\right) = \U(\tau). 
  \end{equation}
\end{enumerate}
\end{lemma}
\begin{proof}
Let $\lambda \in\Sigma$ and $\tau=\T(\lambda)$. 
  It follows from Theorem 6.1, page 184 \cite{Husemoeller}  
that 
$\IC / (\IZ + \tau \IZ)$ and 
$(\EL)_\lambda(\IC)$ are isomorphic  complex tori. 

We have a holomorphic map
given by
\begin{equation*}
z \mapsto
\left[\frac{\wp(z;\tau)-e_1(\tau)}
{e_2(\tau)-e_1(\tau)}:
\frac{\wp'(z;\tau)}{2 (e_2(\tau)-e_1(\tau))^{3/2} }:1\right]
\end{equation*}
if $z\in \IC\ssm( \IZ+\tau \IZ)$ 
and $z\mapsto [0:1:0]$ for $z\in\IZ+\tau\IZ$; the choice of 
root is irrelevant. A straightforward calculation shows
that the image of this holomorphic map lies in
$(\EL)_{\U(\tau)}$. 
 It is classical, that this map induces an isomorphism 
of complex tori between $\IC/(\IZ+\tau\IZ)$ and 
$(\EL)_{\U(\tau)}(\IC)$. 
 Hence the 
$j$-invariants of $(\EL)_\lambda$ and $(\EL)_{\U(\tau)}$
 are equal. In other words,
\begin{equation}
\label{eq:jinvlambda}
  2^8\frac{(\lambda^2-\lambda+1)^3}{\lambda^2(\lambda-1)^2} = 
  2^8\frac{(\U(\tau)^2-\U(\tau)+1)^3}{\U(\tau)^2(\U(\tau)-1)^2},
\end{equation}
by Remark 1.4, page 87 \cite{Husemoeller}.
This equality implies
\begin{equation}
\label{eq:mulambda}
 \U(\tau)\in\left\{\lambda,\frac 1 \lambda,\frac{1}{1-\lambda},
\frac{\lambda}{\lambda-1},
\frac{\lambda-1}{\lambda},1-\lambda \right\}.
\end{equation}

Since $\U\circ T$ is analytic, it suffices to show $\U(\tau)=\lambda$ 
for $\tau = T(\lambda)$ and all sufficiently small $\lambda\in(0,1/2)$
  in order to deduce part (i).
The $j$-invariant $j$
of $(\EL)_\lambda$ is in $(1728,\infty)$ by (\ref{eq:jinvlambda}). 
So there is $x \ge 1$ such that 
 $\tau$ is equivalent to $i x$
under the usual action of $\SL{2}{\IZ}$ on $\IH$. 
It follows from (\ref{def:tau}) and $\lambda\in\IR$ that
 $\tau$ has real part $0$.
Moreover,
  $\lambda < 1/2 < 1-\lambda$ and so $\tau$ has imaginary
 part at least $1$ since $F$ increases on $(0,1)$. In
 particular, $\tau$ is already in the usual fundamental domain of
 the action of $\SL{2}{\IZ}$ on $\IH$, hence
 $\tau= i x$. 
Remark 2, page 251 \cite{Lang:elliptic} gives
 \begin{equation*}
 \U(\tau) = \U(ix) 
  = 16 e^{-\pi x}\prod_{n=1}^{\infty}
 \left(
 \frac{1+e^{-2  \pi x n }}{1+e^{-2  \pi x (n-1/2)}}\right)^8,
 \end{equation*}
 so $\U(\tau) > 0$. Moreover, $\U(\tau) < 16 e^{-\pi x}$
 because  each factor in the
infinite
 product above is in $(0,1)$. 
As $\lambda$ approaches $0$, the $j$-invariant
$j$ goes to $+\infty$. But $j$ is the value of the
 modular $j$-function at $ix$; properties of this function
 imply that $x\rightarrow +\infty$ as $j\rightarrow +\infty$.
Therefore $\U(\tau) \in (0,1/2)$ 
 for $\lambda$ sufficiently small. 
By
 (\ref{eq:mulambda}) the only possibility for $\U(\tau)$ is $\lambda$.
This concludes
 the proof of (i).

For the proof of part (ii)  we remark that
(\ref{eq:modularbusiness1}) implies (\ref{eq:modularbusiness2}) by
definition
of $\Lambda$. By periodicity of the Weierstrass function we get
\begin{alignat*}5
  e_1(\tau+ 2) &= \wp(\tau /2 + 1;\tau + 2) &&=\,& &e_1(\tau), \\
 e_2(\tau +2) &= \wp( 1/2;\tau + 2) &&=\,& &e_2(\tau),\quad\text{and}  \\
e_3(\tau + 2) &= \wp( 1/2+ \tau/ 2 + 1;\tau + 2) &&=\,& &e_3(\tau).
\end{alignat*}
So the first equality in (\ref{eq:modularbusiness1}) holds for all $k$.

 Using (\ref{eq:transfwp}) we derive in
a similar way as above
that 
 the second equality in
(\ref{eq:modularbusiness1}) holds for all $k$. 
\end{proof}


We define the local period map
$\Omega:\Sigma\rightarrow \M{g,2g}{\IC}$  as
\begin{equation}
\label{eq:defineOmegat}
\Omega(\lambda)=
\left[
\begin{array}{ccccc}
\omega_1(\lambda) & \omega_2(\lambda) &  &  &   
\\
  &  &    \ddots & & \\
 &  &   & \omega_1(\lambda) & \omega_2(\lambda) 
\end{array}
\right].
\end{equation}

\subsection{Exponential Map}
\label{sec:exponential}
We take some time to introduce the (local) exponential map
of the Legendre family which will prove useful later on.

By Remark 2, page 251 \cite{Lang:elliptic} we may write
\begin{equation}
\label{def:k}
  e_2(\tau) - e_1(\tau) = r(\tau)^2
\quad\text{where}\quad
r(\tau) = \pi \prod_{n\ge 1} (1-e^{2\pi i n \tau})^2(1+e^{2\pi i
  (n-1/2)\tau})^4
\end{equation}
for any $\tau\in \IH$. The map $r:\IH\rightarrow \IC$ is holomorphic
and non-vanishing

We obtain a holomorphic map
  $\exp:\IC \times \Sigma \rightarrow \IP^2(\IC) \times \Sigma$, the
exponential map, 
given by
\begin{equation*}
(z,\lambda) \mapsto
\left(\left[\frac{\wp(z/\omega_1(\lambda);\T(\lambda))-e_1(\T(\lambda))}
{e_2(\T(\lambda))-e_1(\T(\lambda))}:
\frac{\wp'(z/\omega_1(\lambda);\T(\lambda))}
{2 r(\T(\lambda))^3 }:1\right],\lambda\right)
\end{equation*}
if $z\not\in \IZ+\T(\lambda)\IZ$ 
and $\exp(z,\lambda) = ([0:1:0],\lambda)$ else wise. 

The next lemma summarizes  some basic facts about the exponential map.

\begin{lemma}
\label{lem:exp}
\begin{enumerate}
\item[(i)]
  The diagram  
\begin{equation}
\label{eq:expdiag}
 \bfig\Vtriangle[\IC\times \Sigma` \EL(\IC)`\B(\IC);\exp ` `]\efig
\end{equation}
commutes;  the vertical arrows are projections and the bottom arrow is
the inclusion.
\item[(ii)]
For fixed $\lambda\in \B(\IC)$ the map $\IC \rightarrow
(\EL)_\lambda(\IC)$ given by
$z\mapsto \exp(z,\lambda)$ is a surjective group homomorphism
with kernel $\IZ+\T(\lambda)\IZ$. 
\end{enumerate}
\end{lemma}
\begin{proof}
A straightforward calculation  shows that
$\exp(z,\lambda) \in (\EL)_{\U(\T(\lambda))}(\IC)$
for $z\in\IC$ and $\lambda\in\Sigma$. 
We  already know  $\U(\T(\lambda)) = \lambda$
from Lemma \ref{lem:taulocalinv}(i), 
so the diagram in (\ref{eq:expdiag}) commutes. 
Part (ii) is classical.
\end{proof}

By abuse of notation we write
\begin{equation*}
  \exp : \IC^g\times \Sigma \rightarrow (\AL)_\Sigma 
\end{equation*}
for the fibered product of the exponential map, here
$(\AL)_\Sigma = \piL^{-1}(\Sigma) \subset \AL(\IC)$.  

Let $\xi = (\xi_1,\xi_2)\in (\IR/\IZ)^2$. We  define a holomorphic
map $\widetilde \rho_\xi : \IH \rightarrow \EL(\IC)$
which is needed in Section \ref{sec:torsionpts}.
If $\xi=0$ we set $\widetilde\rho_\xi(\tau) = ([0:1:0],\U(\tau))$. 
If $\xi\not=0$, then we set
\begin{equation*}
  \widetilde \rho_\xi(\tau) = 
\left(\left[\frac{\wp(\xi_1+\tau \xi_2;\tau)-e_1(\tau)}
{e_2(\tau)-e_1(\tau)}:
\frac{\wp'(\xi_1+\tau \xi_2;\tau)}
{2 r(\tau)^3 }:1\right],\U(\tau)\right),
\end{equation*}
this map is well-defined by periodicity of the Weierstrass function. 
We remark, that $\widetilde \rho_\xi(\tau) \in
(\EL)_{\U(\tau)}(\IC)$. 

If $\xi =(\xi_1,\ldots,\xi_{2g}) \in (\IR/\IZ)^{2g}$, then the $g$-fold
 product of $\widetilde \rho_{(\xi_1,\xi_2)},\ldots,\widetilde
\rho_{(\xi_{2g-1},\xi_{2g})}$ is
a holomorphic map $\IH \rightarrow \AL(\IC)$, which by abuse
of notation we also call $\widetilde \rho_\xi$. 
So $\widetilde \rho_\xi(\tau)\in (\AL)_{\U(\tau)}(\IC)$. 

\begin{lemma} 
\label{lem:rho}
Let $\xi=(\xi_1,\ldots,\xi_{2g})\in(\IR/\IZ)^{2g}$.
  \begin{enumerate}
  \item [(i)] If $\lambda\in\Sigma$, then
$\exp(  \Omega(\lambda) \trans{\xi}, \lambda)$ is well-defined and
    \begin{equation*}
      \widetilde\rho_\xi(\T(\lambda) ) = 
\exp(  \Omega(\lambda) \trans{\xi}, \lambda),
    \end{equation*}
where $\trans{}$ means transpose.
   \item[(ii)] [Local monodromy around $0$.] If $\tau\in \IH$, then 
     \begin{equation*}
       \widetilde{\rho}_{\xi}(\tau + 2) = {\widetilde\rho}_{\xi +
         2(\xi_2,0,\xi_4,0,\ldots,\xi_{2g},0)}(\tau).
     \end{equation*}
   \item[(iii)] [Local monodromy around $1$.] If $\tau\in \IH$, then 
     \begin{equation*}
       \widetilde{\rho}_{\xi}\left(\frac{\tau}{-4\tau+1}\right) = {\widetilde\rho}_{\xi - 
         4(0,\xi_1,0,\xi_3,\ldots,0,\xi_{2g-1})}(\tau).
     \end{equation*}
  \end{enumerate}
\end{lemma}
\begin{proof}
  Part (i) follows from the definition of the exponential map,
 $\widetilde \rho_\xi$,  the period matrix
  $\Omega(\lambda)$, and from the equality
$\U(\T(\lambda)) = \lambda$. 

A direct consequence of the
left-hand side of (\ref{def:k}) is
  $r(\tau + 2) = r(\tau)$. 
Part (ii) follows from this, from the first equalities in 
(\ref{eq:modularbusiness1})
and (\ref{eq:modularbusiness2}), and from periodicity of the
Weierstrass function.

We turn to  part (iii).
Say $(\xi_1,\xi_2)\in (\IR/\IZ)^2 \ssm \{0\}$
and let $\tau' = \tau / (- 2\tau+1)$ and $\tau'' = \tau' /(-2\tau'+1)=\tau/(-4\tau+1)$. 
The  functional equalities  (\ref{eq:transfwp}) 
imply 
\begin{equation*}
  \wp\left(\xi_1+\xi_2 \tau'';\tau''\right)
= (- 4\tau+1)^2 \wp(\xi_1 + \tau(-4\xi_1+\xi_2);\tau)
\end{equation*}
and
\begin{equation*}
  \wp'\left(\xi_1+\xi_2 \tau'';\tau''\right)
= (- 4\tau+1)^3 \wp'(\xi_1 + \tau(-4\xi_1+\xi_2);\tau). 
\end{equation*}
Moreover, $e_{1,2}(\tau') = (- 2\tau+1)^2 e_{1,2}(\tau)$ by
(\ref{eq:modularbusiness1}).   Using these, the left-hand side in (\ref{def:k})
implies  $r(\tau') = \chi(\tau) (-2\tau+1) r(\tau)$
with  $\chi(\tau) \in\{\pm 1\}$. But $\chi$ is continuous 
on the connected space $\IH$, so it is a constant $\chi$. 
We derive $r(\tau'') = \chi (-2\tau'+1)r(\tau')= \chi^2
(-4\tau+1)r(\tau) = (-4\tau+1)r(\tau)$. 
Finally, 
$\U(\tau'') =\U(\tau')= \U(\tau)$ by 
(\ref{eq:modularbusiness2}).
Hence $\widetilde\rho_{(\xi_1,\xi_2)}(\tau'') = 
\widetilde \rho_{(\xi_1,-4\xi_1+\xi_2 )}(\tau)$. 
This equality also holds if $(\xi_1,\xi_2)=0$. 
Part (iii) follows since if
 $\xi$ is as in the hypothesis, then $\widetilde\rho_\xi$ is the
product of $\widetilde\rho_{(\xi_1,\xi_2)},\ldots,
\widetilde\rho_{(\xi_{2g-1},\xi_{2g})}$. 
\end{proof}

\subsection{Flat Subgroup Schemes of $\mathcal{A}$}
\label{sec:sbgrpscheme}
In this section we let $\mathcal{E}\rightarrow S$ be as in the
introduction. That is, $\mathcal{E}$ is an abelian scheme
over a non-singular and irreducible quasi-projective curve $S$ defined over
$\IQbar$. Moreover, the fibers of $\mathcal{E}\rightarrow S$ are elliptic
curves. We let
$\mathcal{A}$ denote a $g$-th fibered power of $\mathcal{E}$
over $S$ with $\pi$ the structural morphism $\mathcal{A}\rightarrow S$.

We call a possibly reducible closed subvariety
 $H\subset \mathcal{A}$ a subgroup scheme if it
contains the image of the zero section $S \rightarrow \mathcal{A}$,
if it is mapped to itself by
the inversion $\mathcal{A}\rightarrow \mathcal{A}$ morphism,
and such that the image of $H\times_S H$ under the addition morphism
$\mathcal{A}\times_S \mathcal{A}\rightarrow\mathcal{A}$ is in $H$. 
In this article we disregard standard terminology and required subgroup
schemes  to be reduced.
 This is justified since the base $S$ is a curve over a
field of characteristic $0$.

A subgroup scheme may fail to be flat over $S$; it can  have
horizontal as well as vertical fibers.
We call a subgroup scheme $H$ of $\mathcal{A}$ flat if  its
irreducible components dominate $S$. 
By Proposition III 9.7 \cite{Hartshorne}
this amounts to saying that $\pi|_H:H\rightarrow S$ is flat.

For an integer $N$ we have the multiplication-by-$N$ morphism
$[N]:\mathcal{A}\rightarrow \mathcal{A}$. It is proper since $\pi =
\pi\circ [N]$ is proper.
If $s \in S(\IC)$ and if $X\subset \mathcal{A}$ is Zariski
closed, then  $X_s$ denotes the Zariski closed set
 $\pi|_X^{-1}(s) \subset \mathcal{A}_s$.

The generic fiber of $\mathcal{E}\rightarrow S$ is an elliptic curve
over $\IQbar(S)$, the function field of $S$.
Its $j$-invariant is an element of $\IQbar(S)$. It extends to 
 a morphism $j:S\rightarrow Y(1)$ with
 $j(s)\in\IC$  the $j$-invariant of the elliptic curve
$\mathcal{E}_s$ for all $s\in S(\IC)$. 

\begin{lemma}
\label{lem:jnonconst}
  If $\mathcal{A}$ is not isotrivial, then $j$ is non-constant.
\end{lemma}
\begin{proof}
We consider the $j$-invariant as an element of $\IQbar(S)$ and
 assume it to be constant.
The generic fiber of $\mathcal{E}\rightarrow S$ is an elliptic
curve with $j$-invariant in $\IQbar$. Hence it is isomorphic, over some finite
field extension $K$ of $\IQbar(S)$, to the base change to $K$ of an
 elliptic curve defined over $\IQbar$.
The lemma follows from Theorem 3.1(i) \cite{ZarhinIsogeny}.
\end{proof}

Below
we give a  description of all subgroup schemes of the abelian
scheme  $\mathcal{A}$.

Any  $\varphi=(a_1,\ldots,a_g)\in \IZ^g$ induces a
morphism 
$\varphi: \mathcal{A}\rightarrow \mathcal{E}$ with 
$\varphi(P_1,\ldots,P_g) = [a_1](P_1)+\cdots +
[a_g](P_g)$.
Then $\varphi$ is proper since $\pi=\pi\circ\varphi$ is proper. 
 From now on we identify elements of
$\IZ^g$ with the associated morphism $\mathcal{A}\rightarrow\mathcal{E}$. 
The fibered product  $\Psi = \varphi_1\times_S\cdots\times_S\varphi_r$ of
 $\varphi_1,\ldots,\varphi_r\in\IZ^g$
 determines a proper morphism
$\mathcal{A}\rightarrow \mathcal{E}\times_S \cdots \times_S
\mathcal{E}=\mathcal{B}$ ($r$ factors) over $S$.
The restriction of $\mathcal{A}\rightarrow \mathcal{B}$ to a fiber
above any $s \in S(\IC)$
induces a homomorphism of abelian varieties
$\mathcal{A}_s\rightarrow \mathcal{B}_s$. 
We define the kernel 
  $\ker \Psi $, as
the fibered product of $\Psi:\mathcal{A}\rightarrow\mathcal{B}$ with the zero section
$S\rightarrow\mathcal{B}$. We consider it as a closed subscheme
of $\mathcal{A}$.


\begin{lemma}
\label{lem:horizontal}
\begin{enumerate}
\item[(i)] Any flat subgroup scheme of $\mathcal{A}$ is equidimensional.
\item [(ii)] 
Let us assume that $\mathcal{A}$ is not isotrivial
and let $H\subsetneq \mathcal{A}$ be a flat subgroup scheme.
  There exists $\varphi\in\IZ^g\ssm\{0\}$ such that $H\subset
  \ker\varphi$. 
\item[(iii)] Let $\varphi_1,\ldots,\varphi_r\in\IZ^g$ be $\IZ$-linearly
  independent and $\Psi =
  \varphi_1\times_S\cdots\times_S\varphi_r:\mathcal{A}\rightarrow\mathcal{B}$
  with $\mathcal{B}$ as above. Then $\Psi$ is smooth and
$\ker\Psi$ is a 
  non-singular flat subgroup scheme of $\mathcal{A}$ whose irreducible
  components have dimension
  $g-r+1$. 
\end{enumerate}
\end{lemma}
\begin{proof}
Let $H\subset\mathcal{A}$ be a flat subgroup
scheme with irreducible components $H_1,\ldots,H_l$. 
There exists $s\in S(\IC)$ with the following property. 
For any $1\le i\le l$ there is $P_i\in (H_i\ssm\bigcup_{j\not=i}
H_j)(\IC)$ with $\pi(P_i)=s$.
On applying Exercise II 3.22
 \cite{Hartshorne}, which we call the Fiber Dimension
 Theorem from now on,
we find $\dim_{P_i} (H_i)_s \ge\dim H_i-1$. But equality must hold
because $\pi|_{H_i}:H_i\rightarrow S$ is dominant. By choice of $s$, any
irreducible component of $(H_i)_s$ containing $P_i$ is an irreducible
component of $H_s$. But $H_s$ is an algebraic group and therefore
equidimensional. So $\dim H_s = \dim_{P_i} (H_i)_s = \dim H_i - 1$
is independent of $i$. Part (i) follows.

We turn to part (ii). Since $\mathcal{A}$ is not
 isotrivial, $j$ is non-constant by Lemma \ref{lem:jnonconst} and hence dominant.
The fiber $H_s$ is an
  algebraic subgroup of $\mathcal{A}_s$. We fix $s\in S(\IC)$ with $j(s)$ 
  transcendental. 
Since $H\not=\mathcal{A}$ we may also assume that
 $ H_s \not=\mathcal{A}_s$.
It is a  classical fact that  $\mathcal{E}_s$  does not have complex
  multiplication.  The endomorphism ring of $\mathcal{A}_s
=\mathcal{E}_s^g$ is $\M{g}{\IZ}$. 
 There exists $\varphi\in\IZ^g\ssm\{0\}$ with $H_s \subset \ker\varphi$.
Let us consider $G =
\varphi(H)\subset\mathcal{E}$. Since $\varphi$ is a proper
morphism, $G$ is Zariski closed in $\mathcal{E}$. It is a subgroup scheme of
$\mathcal{E}$. 
But $\mathcal{E}\rightarrow S$ is
also proper, it follows that each irreducible component of $G$ maps surjectively
to $S$. So the fiber of $G\rightarrow S$ above
$s$ meets all irreducible components of $G$. On the other hand,
it contains only the zero element of $\mathcal{E}_s$. 
The Fiber Dimension Theorem,
implies that each irreducible component of $G$ has dimension $1$.
 Therefore, $G_{s'}$ is a finite
group for all $s' \in S(\IC)$. The cardinality can
even be bounded from above  independently of $s'$.
Hence after replacing $\varphi$ by a positive integral
multiple we may assume $G = \varphi(H)$ is the image of the zero
section $S\rightarrow\mathcal{E}$. Hence  $H\subset \ker{\varphi}$
and  (ii) follows.

We now prove (iii). 
The restriction of $\Psi$ to any fiber of $\mathcal{A}\rightarrow S$ induces a
homomorphism between a $g$-th and an $r$-th power of an elliptic curve.
Such a homomorphism is surjective because $\varphi_1,\ldots,\varphi_r$
are linearly independent and have as kernel  an algebraic group of
dimension $g-r$.
These homomorphisms are smooth since domain
and target are abelian varieties over a field of characteristic zero.
Since $\mathcal{A}\rightarrow S$ is flat,  Proposition 17.8.2 \citeEGAIVIV 
implies that $\Psi:\mathcal{A}\rightarrow\mathcal{B}$ is
smooth. Smoothness is preserved under base change, hence
$\ker\Psi\rightarrow S$ is smooth. 
It follows that $\ker\Psi$ is a closed (possibly reducible)
non-singular subvariety
of $\mathcal{A}$. 
We see that $\ker\Psi$ is a subgroup scheme of $\mathcal{A}$. 
But $\ker\Psi\rightarrow S$ is flat and Proposition III 9.7
\cite{Hartshorne} implies that any irreducible component of $\ker\Psi$
dominates $S$. So $\ker\Psi$ is a
flat subgroup scheme of $\mathcal{A}$.  
The statement on the dimension of $\ker\Psi$ follows from Corollary
III 9.6 \cite{Hartshorne}
 and the fact that each fiber
of $\ker\Psi\rightarrow S$ has dimension $g-r$. 
\end{proof}

\section{Torsion Points}
\label{sec:torsionpts}

\subsection{The Main Proposition}

The main result of this section is the following proposition. We
count torsion points on subvarieties of the abelian scheme
 $\AL$ from Section \ref{sec:heights}. In this section we
consider both $\AL$ and $\B$ as defined over $\IC$. 
The cardinality of a set $M$ is denoted by $\# M$.

\begin{proposition}
\label{prop:counting}
Let $X\subset\AL$ be an irreducible closed subvariety
defined over $\IC$ which dominates $\B$  with $\dim X = g$.
Furthermore, we assume that 
\begin{enumerate}
\item[(i)]
either we have $\dim\varphi(X) \ge 2$ for every
$\varphi\in\IZ^g\ssm\{0\}$, 
\item[(ii)] or $\dim X = 1$ and 
$X$ is not an irreducible component of a flat subgroup scheme of $\mathcal{E}$.
\end{enumerate}
Then there exist a (possible reducible) non-singular algebraic curve $C\subset
\AL$ and a  constant $c=c(X) > 0$  such that if $N$ is an
integer with $N\ge c^{-1}$, then
\begin{equation}
\label{eq:propset}
\#\{P\in X(\IC);\,\, Q=[N](P) \in C(\IC) \text{ and }
\dim_Q [N](X)\cap C = 0\} 
  \ge c N^{2g}.
     \end{equation}
\end{proposition}

We remark that in case (i) we have $\dim X \ge \dim \varphi(X) \ge 2$,
so $X$ cannot be a curve.

In the proof we will take $C$ equal to $\ker[T]$ for some positive
integer $T$. So the points $P$ in (\ref{eq:propset}) are torsion.

\subsection{Counting Torsion Points}

In this section we consider $\AL(\IC)$ as a  complex analytic
space \cite{CAS}; throughout the whole paper, all complex analytic spaces are
assumed to be reduced. 
Then $\AL(\IC)$ is even a complex manifold since $\AL$
 is non-singular variety. All references to a topology on
 $\AL(\IC)$ will refer to the Euclidean topology unless stated
 otherwise with the exception that ``irreducible''  refers to the
 Zariski topology. 
 
Recall that $\Sigma=\{z\in \IC;\,\, |z|<1\text{
  and }|1-z|<1\}$.
The preimage $(\AL)_\Sigma = \piL^{-1}(\Sigma)\subset \AL(\IC)$
 is an  open complex submanifold of $\AL(\IC)$. 
We also have a holomorphic (local) exponential map
  $\exp : \IC^g\times \Sigma \rightarrow (\AL)_\Sigma$. 
Its differential is an isomorphism at all points; to see this consider for
example the Jacobian matrix.

Using the topological group $\IR/\IZ$  we now define a continuous function
\begin{equation*}
  \Xi : (\AL)_\Sigma \rightarrow (\IR/\IZ)^{2g}.
\end{equation*}
For any $P\in (\AL)_\Sigma$ there is $w\in \IC^g$ 
such that $\exp(w,\piL(P)) = P$ by Lemma \ref{lem:exp}. 
Because the columns of $\Omega(\piL(P))$ are an $\IR$-basis of $\IC^g$
there exists a unique  $\xi \in\IR^{2g}$ with 
$\Omega(\piL(P))\trans{\xi} = w$. 
We define $\Xi(P)$ to be the image of $\xi$  in $(\IR/\IZ)^{2g}$. 
 We
remark that if $w'\in \IC^g$ with $\exp(w',\piL(P)) = P$, then the
resulting $\xi'$ will differ from $\xi$ by an element in
$\IZ^{2g}$.
So $\Xi(P)$ is well-defined.
 It remains to show that
$\Xi$ is continuous.
Indeed, $P$ has an
open neighborhood $U$ in $(\AL)_\Sigma$ such that there exists a holomorphic map 
$\log: U \rightarrow \IC^g \times \B(\IC)$ with
$\log(U)$  open in $\AL(\IC)$ and with
$\exp\circ\log$  the identity on $U$. Let $\log^0: U\rightarrow
\IC^g$ denote $\log$ composed with the projection onto $\IC^g$. 
The matrix 
\begin{equation*}
  \left[
\begin{array}{cc}
\Omega(\piL(P)) \\
\overline{\Omega}(\piL(P))
\end{array}\right]
\end{equation*}
is invertible because the columns of $\Omega(\piL(P))$ are
$\IR$-linearly independent;  the bar denotes complex conjugation.
We set
\begin{equation*}
  \widetilde \xi(P)
=\left[
\begin{array}{cc}
\Omega(\piL(P)) \\
\overline{\Omega}(\piL(P))
\end{array}\right]^{-1}
\left[
\begin{array}{cc}
\log^0(P) \\ \overline{\log^0}(P)
\end{array}
\right] \in \IR^{2g}.
\end{equation*}
Then $\widetilde \xi$ is clearly continuous on $U$. We let
 $\Xi|_U$ denote  $\widetilde\xi:U \rightarrow\IR^{2g}$ composed with the natural
map $\IR^{2g}\rightarrow (\IR/\IZ)^{2g}$.

By Lemma \ref{lem:exp}(ii) the map
$\Xi|_{(\AL)_\lambda}:(\AL)_\lambda(\IC)\rightarrow(\IR/\IZ)^{2g}$ 
is a group isomorphism for all $\lambda\in\Sigma$. 





For any variety $X$ defined over $\IC$ let $\ns{X}$ denote its Zariski open and dense subset
of non-singular points. If $X\subset \AL$ is a (possibly reducible)
  subvariety we set $X_{\Sigma} = (\AL)_\Sigma \cap X(\IC)$.

\begin{lemma}
\label{lem:imageopen}
 Let $X$ be an irreducible closed subvariety of $\AL$ of
 dimension $g$.
Let  $P\in \ns{X}_\Sigma$ 
such that $\Xi|_{X_\Sigma}^{-1}(\Xi(P))$ contains a countable
neighborhood of $P$.
Then $P$ is isolated in $\Xi|_{X_\Sigma}^{-1}(\Xi(P))$
and $\Xi(X_\Sigma)$ contains a non-empty open subset of
$(\IR/\IZ)^{2g}$. 
\end{lemma}
\begin{proof}
Let $\log^0:U\rightarrow \IC^g$ and $\widetilde\xi:U\rightarrow \IR^{2g}$ be
as in the proof of continuity of $\Xi$ where $U$ is an open
neighborhood of $P$. We may assume $X(\IC)\cap U \subset
\ns{X}_\Sigma$; so, $X(\IC)\cap U$ is a  complex manifold
of dimension $g$.

  We define
  \begin{equation}
\label{eq:defineZ}
    Z = \{(Q,w)\in (X(\IC)\cap U)\times\IC^{2g};\,\, 
\Omega(\piL(Q)) w - \log^0(Q) =  0\}.
  \end{equation}
Then $Z$, being the set of common zeros of holomorphic functions, is an
analytic subset of the complex manifold $(X(\IC)\cap U)\times \IC^{2g}$. 
It contains
$(P,\xi)$ where $\xi=\widetilde\xi(P)$. 

Using the Jacobian condition and the fact that the matrix
$\Omega(\piL(Q))$ has rank $g$ we conclude that $Z$ is a
 complex submanifold of $(X(\IC)\cap U)\times\IC^{2g}$
of dimension $2g$.

Now let $q:Z\rightarrow \IC^{2g}$ denote the projection onto the last
$2g$ coordinates. 
One readily checks that
\begin{equation*}
  \Xi|_{X(\IC)\cap U}^{-1}(\Xi(P)) \times \{\xi\} = q^{-1}(\xi)
 = q^{-1}(q(P,\xi)). 
\end{equation*}
By hypothesis, there is a countable neighborhood of $(P,\xi)$ 
in $q^{-1}(q(P,\xi))$. But latter is a complex analytic subset of $Z$. So it contains
$(P,\xi)$ as an isolated point. Moreover, $P$ is isolated
in $\Xi|_{X_\Sigma}^{-1}(\Xi(P))$ and the first assertion follows.

We note that $Z$ and $\IC^{2g}$ are complex manifolds 
of equal dimension.
The comment on page 64 \cite{CAS} implies that $q$ is a finite
holomorphic map at
$(P,\xi)$.  The proposition on page 107 in the same reference tells us
that $q$ is open at $(P,\xi)$. 

In particular, $q(Z)$ contains an open
neighborhood $W'\subset \IC^{2g}$ of $\xi$. Since $\xi\in\IR^{2g}$ it
follows that $W'\cap\IR^{2g}$ is an open and non-empty subset of
$\IR^{2g}$. 
Finally, $W$, its image in $(\IR/\IZ)^{2g}$,
is open too.

By definition of $\Xi$ and (\ref{eq:defineZ}),  any element of $W$ is the image
under $\Xi$ of some element of $X(\IC)\cap U \subset X_\Sigma$,
so the second assertion follows.
\end{proof}

We need a simple counting result on $N$-th roots of elements
in $(\IR/\IZ)^{2g}$.

\begin{lemma}
\label{lem:count}
Let $W$ be a non-empty open subset of $(\IR/\IZ)^{2g}$ 
then there exists a constant $c>0$ with the following property.
If $\xi_0 \in (\IR/\IZ)^{2g}$
and if $N$ is an integer with 
$N\ge c^{-1}$, then 
 \begin{equation*}
\#\{\xi\in W;\,\, N\xi = \xi_0 \} \ge cN^{2g}.
 \end{equation*}
\end{lemma}
\begin{proof}
There are $x_1,\ldots,x_{2g}\in \IR$ and an
$\epsilon \in (0,1/2]$
such that the image of $U=\prod_{i=1}^{2g}(x_i-\epsilon,x_i+\epsilon)$
under the natural map
$\IR^{2g}\rightarrow (\IR/\IZ)^{2g}$ is contained in $W$.
Let $y=(y_1,\ldots,y_{2g})\in \IR^{2g}$ be a lift of $\xi_0$
and $N$ an integer with $N\ge\epsilon^{-1}$.
We have
$NU-y = \prod_{i=1}^{2g} (N x_i-\epsilon N-y_i,N x_i+\epsilon N-y_i)$. 
Any real interval $(a,b)$ contains at least $b-a-1$ integers. Hence
$\# (NU-y)\cap \IZ^{2g} \ge (2\epsilon N - 1)^{2g}
\ge (2\epsilon N - \epsilon N)^{2g}=
 \epsilon^{2g} N^{2g}$. 
If $Nu-y\in \IZ^{2g}$ with $u\in U$, then $N \xi = \xi_0$
where $\xi\in W$ is the image of $u$ in $(\IR/\IZ)^{2g}$. 
If $u,u'\in U$ have equal image in $(\IR/\IZ)^{2g}$, 
then $u=u'$ because $\epsilon \le 1/2$.  So
$\#\{\xi\in W;\,\,N\xi=\xi_0\} 
\ge \epsilon^{2g} N^{2g}$ and the current lemma holds
with $c=\epsilon^{2g}$.
\end{proof}

The following remark on
$\Xi:(\AL)_\Sigma\rightarrow (\IR/\IZ)^{2g}$  will be
useful further down. 
By abuse of notation we also use $\Xi$ to denote
the continuous map $(\EL)_\Sigma\rightarrow (\IR/\IZ)^2$ if $g=1$. 
For any $\varphi=(a_1,\ldots,a_g)\in\IZ^g$ we
 have a commutative diagram
\begin{equation}
\label{eq:Xidiag}
\bfig
\square[(\AL)_\Sigma`(\IR/\IZ)^{2g}`(\EL)_\Sigma`(\IR/\IZ)^{2};
\Xi`\varphi|_{(\AL)_\Sigma}` \Xi(\varphi)`\Xi]
\efig
\end{equation}
where 
\begin{equation*}
 \Xi(\varphi)(\xi_1,\ldots,\xi_{2g}) = 
(a_1 \xi_1+ a_2 \xi_3 +\cdots + a_g \xi_{2g-1},
a_1 \xi_{2}+a_2 \xi_4 + \cdots +a_g \xi_{2g}) 
\end{equation*}
is a continuous homomorphism of groups.

By the next lemma, Proposition \ref{prop:counting} holds for hypersurfaces
which satisfy a  non-degeneracy property with respect to $\Xi$.

\begin{lemma}
\label{lem:propnondegen}
Let $X$ be an irreducible closed subvariety of
$\AL$ of dimension $g$.
Let us assume that there exists $P\in \ns{X}_\Sigma$ which is isolated in 
 $\Xi|_{X_\Sigma}^{-1}(\Xi(P))$.
Then  Proposition
\ref{prop:counting} holds
for $X$.
\end{lemma}
\begin{proof}
 
The generic fiber 
$(\AL)_\eta$ of $\AL\rightarrow \B$
is an abelian variety over the rational function field $\IC(\B)$. It is
 the $g$-th power  of  $(\EL)_\eta$, the generic fiber of
 $\EL\rightarrow \B$. Then $(\EL)_\eta$ is an elliptic curve with non-constant
$j$-invariant, cf. the left-hand side of (\ref{eq:jinvlambda}).
 So $(\EL)_\eta$ cannot have
 complex multiplication. Since
$(\AL)_\eta =(\EL)_\eta^g$ we may identify 
the group of homomorphisms of algebraic groups
 $(\AL)_\eta\rightarrow (\EL)_\eta$ 
with $\IZ^g$. Moreover, when replacing  $\IC(\B)$ by an algebraic
closure $K$ and regarding $(\AL)_\eta$ and $(\EL)_\eta$ over
this larger field we do not get any new homomorphisms
$(\AL)_\eta\rightarrow(\EL)_\eta$.

We recall that the Manin-Mumford Conjecture, a result of Raynaud \cite{Raynaud:MM},
holds for an abelian variety defined over any field of characteristic
$0$. As a consequence 
 there exist  finitely 
many homomorphisms
$\varphi_{1},\ldots,\varphi_{n}\in\IZ^g\ssm\{0\}$
 with the following property.
Any torsion  point of $(\AL)_\eta(K)$ in $X_\eta(K)$
 is contained in the kernel of some
$\varphi_{i}$.

As usual, we consider each $\varphi_i$ as a homomorphism
$\AL\rightarrow \EL$.
We may pick $\xi_0 \in (\IQ/\IZ)^{2g}\subset (\IR/\IZ)^{2g}$ which avoids
$\ker\Xi(\varphi_1)\cup\cdots\cup\ker\Xi(\varphi_n)$. 
Then $\xi_0$ is torsion of order
$T$, say. We shall apply Lemma \ref{lem:count} to $W$ and $\xi_0$. But
first we define $C$ to be $\ker[T]$. This is just the
$T$-torsion subgroup scheme of $\AL$. It is a possibly
reducible non-singular curve 
whose irreducible components dominate $\B$
by  Lemma \ref{lem:horizontal}(iii).

We note that $P$ as in statement of this lemma satisfies the
hypothesis of Lemma \ref{lem:imageopen}. 
So $\Xi(X_\Sigma)$ contains a non-empty open subset 
 $W\subset(\IR/\IZ)^{2g}$.
Let $c>0$ be as in Lemma \ref{lem:count} and $N \ge c^{-1}$. Recalling
$W\subset \Xi(X_\Sigma)$, this lemma
provides at least $cN^{2g}$ points $P'\in X(\IC)$ such that $N\Xi(P') =
\xi_0$. We proceed to show that any such $P'$ lies in the set on the left
of (\ref{eq:propset}). 

We have $0=T\xi_0 = TN\Xi(P') = \Xi([TN]P')$ and 
because $\Xi$ is fiberwise a group isomorphism 
this shows
$Q\in C(\IC)$ for $Q=[N](P')$. We note $\Xi(Q) =\xi_0$. It remains to show 
\begin{equation}
\label{eq:Qisolated}
\dim_Q [N](X)\cap C=0.  
\end{equation}
Let us assume the contrary and suppose that $Z'\subset
[N](X)\cap C$ is an irreducible curve containing $Q$. 
Then $Z'$ dominates $\B$ since it is an irreducible component of $C$.
We may choose an irreducible closed subvariety $Z\subset X$
with $[N](Z) = Z'$. Then  $\dim Z = 1$
by the Fiber Dimension Theorem  since $[N]$ has finite fibers. 
Let $P''\in Z(\IC)$ with $[N](P'')=Q$.
We have  $[TN](Z) = [T](Z') \subset [T](C) = \epsilon_L
(\B)$. Certainly, $Z$ dominates $\B$ 
so the function field of $Z$ contains the function field of $\B$.
The former leads to a torsion point in
$X_\eta(K) \subset (\AL)_\eta(K)$. By
our setup, this torsion point is in the kernel of some
$\varphi_{i}$. This implies 
$Z\subset \ker\varphi_i$
 and in particular $\varphi_i(P'')=0$. 
Diagram (\ref{eq:Xidiag}) implies
$\Xi(\varphi_i)(\Xi(P'')) = 0$. 
If we multiply this equality with $N$ we get
$0=\Xi(\varphi_i)(\Xi([N](P'')))=
\Xi(\varphi_i)(\Xi(Q))= \Xi(\varphi_i)(\xi_0)$. But this contradicts our choice
of $\xi_0$. So (\ref{eq:Qisolated}) must hold true. 
\end{proof}

\subsection{The Degenerate Case}
Throughout this section 
 $X\subset\AL$ will be an irreducible closed subvariety which
dominates $\B$. In general we pose no restriction on the dimension
of $X$. 

The previous discussion, in particular Lemma \ref{lem:propnondegen},
 suggests that we  study the following property more careful.
We call $P\in \ns{X}_\Sigma$  degenerate if
it is not isolated in
 $\Xi|_{X_\Sigma}^{-1}(\Xi(P))$. 
If all points of $\ns{X}_\Sigma$ are degenerate, then  we call $X$ degenerate. 

In order to handle the degenerate case we exploit monodromy of the family
$\AL\rightarrow \B$ using the holomorphic map
$\widetilde \rho_\xi$ from Section \ref{sec:exponential}.


We say that $\xi=(\xi_1,\ldots,\xi_{2g})\in (\IR/\IZ)^{2g}$ is in
general position if
$a_1\xi_1+\cdots +a_{2g}\xi_{2g}\not=0$ in $\IR/\IZ$ for all
$(a_1,\ldots,a_{2g})\in\IZ^{2g}\ssm\{0\}$. 

We define $\widetilde\Sigma = \T(\Sigma)\subset \IH$. This set is
 open in
$\IH$ because $\T$ is holomorphic and non-constant. 
We remark that Lemma \ref{lem:taulocalinv} implies
$\U(\widetilde \Sigma) = \Sigma$.

The next lemma uses Kronecker's Theorem in diophantine approximation
 to extract information on
 the image of $\widetilde\rho_\xi$ for  $\xi$ in general position.
This image is what is refered to colloquially as an analytic subgroup
scheme in the introduction.

\begin{lemma}
\label{lem:applykronecker}
Let $\xi=(\xi_1,\dots,\xi_{2g})\in(\IR/\IZ)^{2g}$ such that $(\xi_2,\xi_4,\ldots,\xi_{2g})\in
(\IR/\IZ)^g$ is in general position.
 Then $\widetilde\rho_\xi(\IH)$
is Zariski dense in $\AL$.
\end{lemma}
\begin{proof}
Say $Z$ is the Zariski closure of
  $\widetilde\rho_\xi(\IH)$ in $\AL$. 

We fix $\tau\in \widetilde\Sigma \subset \IH$; then $\tau =
\T(\lambda)$ for some $\lambda\in\Sigma$. 
By hypothesis we have $\widetilde\rho_\xi(\tau+2k)\in Z(\IC)$ for all
$k\in\IZ$.
We apply Lemma \ref{lem:rho}(ii) by induction to obtain
\begin{equation*}
 \widetilde\rho_\xi(\tau+2k)=\widetilde\rho_{\xi+2k(\xi_2,0,\ldots,\xi_{2g},0)}(\tau)
\in Z(\IC).
\end{equation*}

Kronecker's Theorem IV page 53 \cite{Cassels} and our
hypothesis on $\xi$ imply that 
\begin{equation*}
 \{(2k\xi_2,2k\xi_4,\ldots,2k\xi_{2g});\,\,k\in\IZ\}\quad\text{lies
   dense in}
\quad(\IR/\IZ)^g.  
\end{equation*}

Now
\begin{equation*}
  f(z_1,\ldots,z_g) = \exp(\Omega(\lambda)
\trans{( z_1,\xi_2,z_2,\xi_4,\ldots,z_g,\xi_{2g})},\lambda)
\end{equation*}
is a $\IZ^{g}$-periodic holomorphic map
$f:\IC^g\rightarrow\AL(\IC)$.
We recall Lemma \ref{lem:rho}(i) to deduce that
 $f$ takes values in
$Z(\IC)$ at $(2k\xi_2,2k\xi_4,\ldots,2k\xi_{2g})$ for all $k\in\IZ$. 
By continuity we conclude $f(\IR^g) \subset Z(\IC)$. 
So
$f^{-1}(Z(\IC))$ is a complex analytic subset of $\IC^g$ which contains
$\IR^g$. But the only such  set is $\IC^g$ itself. Hence
$f(\IC^g)\subset Z(\IC)$. 
From the definition (\ref{eq:defineOmegat}) of $\Omega$ 
and because $\omega_1$ never vanishes we see 
$f(\IC^g)= (\AL)_{\lambda}(\IC)  \subset Z(\IC)$.

We let $\tau$ vary over  $\widetilde\Sigma$ and use
 $\U(\widetilde\Sigma) = \Sigma$ to find
$(\AL)_\Sigma \subset Z(\IC)$. Now $(\AL)_\Sigma$ is non-empty and  open in 
 $\AL(\IC)$ with respect to the Euclidean topology. It is therefore
Zariski dense in $\AL$ and thus $Z=\AL$.
\end{proof}

Without much effort one can strengthened this argument
to show that $\widetilde\rho_{\xi}(\IH)$ is
not contained in a proper analytic subset of $\AL(\IC)$.

The next lemma uses the fact that $\B$ has dimension $1$ in an
essential way. 

\begin{lemma}
\label{lem:notindep}
If $P\in \ns{X}_\Sigma$   is 
degenerate,
 then 
$\widetilde \rho_{\Xi(P)} (\IH) \subset X(\IC)$.
\end{lemma}
\begin{proof}
For brevity we set $\xi=\Xi(P)$.
By the first assertion of Lemma \ref{lem:imageopen} the fiber $Z =
\Xi|_{X_\Sigma}^{-1}(\xi)$ is uncountable. 

We claim that $\Xi^{-1}(\xi)\subset \widetilde\rho_\xi(\widetilde
\Sigma)$. Indeed, say $P'\in(\AL)_\Sigma$ with $\Xi(P')=\xi$. 
By Lemma \ref{lem:taulocalinv}(i) 
we have  $\U(\tau)=\piL(P')$ with $\tau = \T(\pi_L(P'))$. Hence
$P'=\exp(\Omega(\piL(P'))\xi,\piL(P')) = \widetilde
\rho_\xi(\tau)$ by Lemma \ref{lem:rho}(i), as desired.

The claim implies  $Z\subset \widetilde \rho_\xi(\widetilde
\Sigma)$. 
So $Y = \widetilde\rho^{-1}_\xi(X(\IC))$ is an uncountable
complex analytic subset of $\IH$.
In particular, there is $\tau\in Y$ with $\dim_\tau Y \ge 1$. 

Now we make use of the trivial, but crucial, fact that $\IH$ has
dimension $1$
at all points. 
So we must have $\dim_\tau Y = \dim_\tau\IH$.
With this, the  Identity Lemma, page 167 \cite{CAS} implies $Y=\IH$.
In other words, $\widetilde\rho_\xi(\IH)\subset X(\IC)$.
\end{proof}

The following possibly well-known
statement helps to study the degenerate case.

\begin{lemma}
\label{lem:translate}
  Let $A$ be an abelian variety of dimension $g$
 and $Y\subset A$ an irreducible closed
  subvariety. We let $s:A^g\rightarrow A$ denote the morphism
$(P_1,\ldots,P_g)\rightarrow P_1+\cdots +P_g$. Then $s(Y^g)$ is the
  translate of an abelian subvariety of $A$. 
\end{lemma}
\begin{proof}
After translating $Y$ we may assume $0\in Y(\IC)$. We  also
immediately reduce to the case $Y\not=0$ and $s(Y^g)\not= A$. 
For $1\le k\le g$ we let $s_k : A^k\rightarrow A$ denote the morphism
$(P_1,\ldots,P_k)\rightarrow P_1+\cdots + P_k$. 
Then $s_1(Y) \subset s_2(Y^2) \subset\cdots\subset s_g(Y^g)$ because
$0\in Y(\IC)$. Each $s_k(Y^k)$ is an irreducible closed subvariety
of $A$.
The dimensions satisfy
$1\le \dim s_1(Y) \le\cdots\le \dim s_g(Y^g) \le g-1$. By the
Pigeonhole Principle there exist $k < l$ such that 
$\dim s_k(Y^k) = \dim s_l(Y^l)$, we may assume $l=k+1$. We must even have 
$s_k(Y^k) = s_{k+1}(Y^{k+1}) = B$ because these varieties
are irreducible. Certainly, $0\in B(\IC)$ and if 
 $P_i,Q_i\in Y(\IC)$ for $1\le i\le k$,
then 
$(P_1+\cdots + P_k + Q_1)+Q_2+\cdots +Q_k \in B(\IC)+
Q_2+\cdots + Q_k$. By induction we get
$P_1+\cdots + Q_k\in B(\IC)$, so $B$ is closed under
addition.
If $P\in B(\IC)$ is fixed, then $Q\mapsto Q+P$ defines a
proper morphism $B\rightarrow B$ with finite fibers. 
Comparing dimension we see that this morphism is surjective, so there
is $Q\in B(\IC)$ with $P+Q=0$.  Hence $B$ is closed under
inversion too. Therefore, $B$ is an abelian subvariety of $A$.
It follows that $s(Y^g)=B$.
\end{proof}

\begin{lemma}
\label{lem:degenerate}
Let us assume that $X$ is degenerate and 
$X\not=\AL$. 
There exist $\varphi\in\IZ^g\ssm \{0\}$ and an uncountable subset $\Sigma'\subset \Sigma$
 such that for any $\lambda\in \Sigma'$
there is
an irreducible component $X'_\lambda$ of $X_\lambda$ with
$\dim \varphi(X'_\lambda) = 0$.
\end{lemma}
\begin{proof}
Lemmas \ref{lem:applykronecker} and \ref{lem:notindep} 
 imply that if $P\in\ns{X}_\Sigma$, then $\Xi(P)$ is not in general position.
If $a=(a_1,\ldots,a_{2g})\in \IZ^{2g}$ we let
$G_a$ denote the closed subgroup $\{(\xi_1,\ldots,\xi_{2g})\in (\IR/\IZ)^{2g};\,\,
a_1\xi_1+\cdots +a_{2g}\xi_{2g} = 0 \} \subset (\IR/\IZ)^{2g}$. Then
\begin{equation*}
  X_\Sigma = (X\ssm\ns{X})_\Sigma\cup
 \bigcup_{a\in\IZ^{2g}\ssm\{0\}} \Xi^{-1}(G_a)\cap X_\Sigma.
\end{equation*}
Each of the countably many sets in the union above is 
closed in $X_\Sigma$.
Of course, $X_\Sigma$ is non-empty because $X$ dominates $\B$. 
So the interior in $X_\Sigma$ of one of the sets
\begin{equation*}
  (X\ssm\ns{X})_\Sigma, \quad \Xi^{-1}(G_a)\cap X_\Sigma \text{ with
  }a\in\IZ^{2g}\ssm \{0\}
\end{equation*}
is non-empty by the Baire Category Theorem. But it cannot be 
 $(X\ssm \ns{X})_\Sigma$  since the singular locus
 is a Zariski closed and proper subset of $X$. So 
 there is $a\in\IZ^{2g}\ssm\{0\}$ and a non-empty open
set $U\subset X_\Sigma$ with 
\begin{equation}
\label{eq:XisubsetHa}
\Xi(U)\subset G_a.
\end{equation}

We remark that $\piL|_{X(\IC)}:X(\IC)\rightarrow \B(\IC)$ is a
non-constant holomorphic function and $\B(\IC)\subset \IC$ is open. So
$\piL(U)$ is open in $\B(\IC)$ by the corollary on page 109
\cite{CAS}. 
Let $\Sigma'$  be the set of transcendental elements in $\piL(U)$. 
Then $\Sigma'$ is uncountable 
and if $\lambda\in\Sigma'$
then $(\EL)_\lambda$ does not have complex
multiplication because its $j$-invariant
is the transcendental number (\ref{eq:jinvlambda}).

Let $V\subset \IR^{2g}$ be the preimage of $G_a$ under
$\IR^{2g}\rightarrow (\IR/\IZ)^{2g}$.
Say $\lambda\in\Sigma'$ and 
let $X'_\lambda$ be an irreducible component of the fiber $X_\lambda$ 
 with $U_\lambda = X'_\lambda(\IC)\cap U\not=\emptyset$.
It follows from
(\ref{eq:XisubsetHa}) that $U_\lambda\subset
\exp(\Omega(\lambda)V,\lambda)$. 
In fact, we even have
$P_1+\cdots+P_g \in \exp(\Omega(\lambda)V,\lambda)$ for 
$P_1,\ldots,P_g\in U_\lambda$ since $V$ is a group. 
In the notation of Lemma
\ref{lem:translate} we can restate this as 
\begin{equation}
\label{eq:fUgsubset}
 s(U_\lambda^g)\subset \exp(\Omega(\lambda)V,\lambda). 
\end{equation}

We claim that $s({X'_\lambda}^g)\not= (\AL)_\lambda$. This will
complete the proof of the lemma in view of Lemma \ref{lem:translate}
and the following argument.
Any proper algebraic subgroup of $(\AL)_\lambda$ is
 in the kernel of a non-trivial homomorphism
$(\AL)_\lambda\rightarrow (\EL)_\lambda$ which can be identified with
an element $\varphi_\lambda \in \IZ^{g}\ssm\{0\}$ since $(\EL)_\lambda$ lacks
complex multiplication. 
Moreover, since $\IZ^g\ssm\{0\}$ is countable
we may  assume that $\varphi_\lambda$ is independent of $\lambda$
after  replacing $\Sigma'$ by an
uncountable subset.

To prove our claim let us suppose 
$s({X'_\lambda}^g) = (\AL)_\lambda$ and derive
contradiction. The set $U_\lambda^g$ is non-empty and open in
${X'_\lambda}^g(\IC)$,  so it  lies Zariski dense. Hence it 
contains a point where the variety ${X'_\lambda}^g$ is non-singular and 
where $s|_{{X'_\lambda}^g}$ has maximal rank. 
The holomorphic map  $s|_{{X'_\lambda}^g(\IC)}$ is open at this point. 
Thus $s(U_\lambda^g)$ contains a non-empty  open subset of 
$(\AL)_\lambda(\IC)$. So (\ref{eq:fUgsubset}) implies that 
$V$ contains a non-empty open  subset of $\IR^{2g}$,
a contradiction. 
\end{proof}

The statement of this lemma is void if $\dim X=1$. Indeed,
in this case every irreducible component of $X_\lambda$ is a point
and, as such, the
 translate of the trivial abelian subvariety of
$(\AL)_\lambda$. 




\begin{lemma}
\label{lem:degenerate2}
Let us assume that $X$ is degenerate and $X \not=\AL$.  There exists
$\varphi\in\IZ^g\ssm\{0\}$ such that 
$\dim \varphi(X) \le 1$.
\end{lemma}
\begin{proof}
Let $\varphi,\Sigma',$ and $X'_\lambda$ be as in Lemma \ref{lem:degenerate}.
 The 
 Fiber Dimension Theorem applied to 
$\piL|_X:X\rightarrow \B$ implies 
$X'_\lambda \ge \dim X-1$ for all $\lambda\in \Sigma'$. So, the fiber of
$\varphi|_X:X\rightarrow \EL$ through any point of
$ \bigcup_{\lambda\in \Sigma'} X'_\lambda(\IC)$ has dimension at least $\dim
X-1$. 
By comparing dimensions we see that
 $\bigcup_{\lambda\in \Sigma'} X'_\lambda(\IC)$ is Zariski dense in $X$. We again
apply the Fiber Dimension Theorem to conclude that
there is $P\in \bigcup_{\lambda\in \Sigma'}X'_\lambda(\IC)$ such that
$\dim_P \varphi|_X^{-1}(\varphi(P)) = \dim X - \dim \varphi(X)$. So
$\dim\varphi(X)\le 1$, as desired.
\end{proof}

\subsection{The Case of Curves}
\label{sec:curves}

We first  handle 
case (ii) of Proposition \ref{prop:counting}.
The case of curves will be the starting point of an inductive argument
eventually leading to the proof of Theorem \ref{thm:main}.

\begin{proof}[Proof of Proposition \ref{prop:counting}(ii)]
Let $P\in \ns{X}_\Sigma$ and $\lambda = \pi_L(P)$.
In the current case
 $X\subset \EL$ is a curve which dominates $\B$. So $X_\lambda(\IC)$ is
finite of cardinality bounded independently of $P$.
By Lemma \ref{lem:propnondegen} we may assume that
$P$ is not isolated
 in  $ \Xi|_{X_\Sigma}^{-1}(\Xi(P))$.
So Lemma \ref{lem:notindep} implies
$\widetilde\rho_\xi(\IH) \subset X(\IC)$. 
We set $\xi = (\xi_1,\xi_2)= \Xi(P) \in (\IR/\IZ)^2$.

First we use local monodromy around  $0$ 
by applying  Lemma \ref{lem:rho}(ii) to see
\begin{equation*}
  \widetilde\rho_{\xi + 2k(\xi_2,0)}(\T(\lambda))
 = \widetilde\rho_{\xi}(\T(\lambda)+2k) \in X_\lambda(\IC)\quad\text{for
    all}\quad
k\in\IZ.
\end{equation*}
By Lemma  \ref{lem:rho}(i) and 
the Pigeonhole Principle there is an integer $N\ge 1$ independent
of $P$ with $N \xi_2 = 0$.


 To handle $\xi_1$  we need local monodromy around the cusp
$1$.
We use Lemma \ref{lem:rho}(iii) and obtain
\begin{equation*}
  \widetilde\rho_{\xi - 4k(0,\xi_1)}(\T(\lambda)) \in X_\lambda(\IC)\quad\text{for
    all}\quad
k\in\IZ.
\end{equation*}
As before we have $N \xi_1=0$, after possibly adjusting $N$.

Because $N$ is independent of $P$ we obtain
\begin{equation*}
 \widetilde\rho_\xi (\T(\pi_L(P))) \in X\cap \ker[N] \quad\text{for
   all}\quad
 P\in\ns{X}_\Sigma.
\end{equation*}
Therefore, $X \cap \ker[N]$ is
infinite and so $X\subset \ker [N]$. This is a contradiction
since $\ker[N]$ is one-dimensional  flat subgroup scheme
of $\mathcal{E}$ by Lemma \ref{lem:horizontal}(iii).
\end{proof}

\subsection{Proof of the Proposition} 
Let $X$ be as in the statement
of the proposition. 
Part (ii) was already proved in Section \ref{sec:curves} and 
it remains to show (i).
The hypothesis  and Lemma \ref{lem:degenerate2} imply
that $X$ is not degenerate. The proof now follows from
Lemma \ref{lem:propnondegen}.
\qed

\section{Intersection Numbers}
\label{sec:internumb}

In this section we use Proposition \ref{prop:counting} with a theorem
of Siu to construct an
  auxiliary non-zero global section of a certain
 line bundle. We then deduce the following height inequality.

\begin{proposition}
\label{prop:heightlb}
  Let $X\subset\AL$ be an irreducible closed subvariety
  defined over $\IQbar$ which dominates $\B$, has dimension $g$, and 
satisfies (i) or (ii) of Proposition \ref{prop:counting}.
There exists a constant $c > 0$ with the following property. For any
integer $N \ge c^{-1}$ there is a non-empty Zariski open subset
$U\subset X$ and a constant $c'(N)$ such that
\begin{equation*}
  \heightlb{[2^N](P)}{\AL} \ge c 4^N \heightlb{P}{\AL} -c'(N)
\end{equation*}
for all $P \in U(\IQbar)$. 
\end{proposition}

\subsection{Degree and Height Lower Bounds}

The proof  Proposition \ref{prop:heightlb} is based on a
degree estimate.

Let $f$ be a rational map between two irreducible varieties. Then $\dom{f}$ denotes
the domain of $f$. 
If source and target of $f$  have equal 
 dimension
we  define $\deg{f}$, the degree of $f$, as follows. 
If $f$ is dominant, then
$\deg{f}$
is the degree of the (finite) extension of function fields induced by
$f$. If $f$ is not dominant, we set $\deg{f}=0$.

\begin{lemma}
\label{lem:degree}
Let $X$ be an irreducible variety defined over $\IC$ of dimension $g$ and let
$f:X\dashrightarrow \IP^g$  be a rational map. If $Q\in\IP^g(\IC)$,
then the 
number of zero-dimensional irreducible components of
 $f^{-1}(Q)$ is at most $\deg{f}$.
\end{lemma}
\begin{proof}
In this proof, any mention to a topology on $X(\IC)$ or $\IP^g(\IC)$
refers to the Euclidean topology if not stated otherwise.

Let $P_1,\ldots,P_d\in \dom{f}(\IC)$ be distinct and 
isolated in the fiber of $f$ above
$Q$ with respect to the Zariski topology. 
The $P_i$ are also isolated with respect to the Euclidean topology.
We may assume $d\ge 1$.


We regard $\dom{f}(\IC)$ and $\IP^g(\IC)$ as $g$-dimensional 
 complex analytic spaces and
$f$ as a holomorphic map between them. 
The comment on page 64
\cite{CAS} implies that $f$ is a finite holomorphic map at  $P_i$ for
$1\le i\le d$. By the
proposition on page 107 \cite{CAS} we conclude that $f$ is an open
 map at each $P_i$.
Hence there exists an open
neighborhood $U_i$ of $P_i$ in $\dom{f}(\IC)$ such
that $f|_{U_i}$ is an open mapping.
We may assume that the $U_i$ are pairwise disjoint. 
The intersection $W = \bigcap_{i=1}^d f(U_i)$ is 
open  in $\IP^g(\IC)$ and  contains $Q$. 
Let $Q'\in W$.  There exists $P'_i \in U_i$ with
$f(P'_i)=Q'$.
The resulting $P'_i$ are pairwise distinct, so
the fiber of $f$ above any point in $W$
has cardinality at least $d$. 

There exists a Zariski closed and proper $Z\subset \IP^g$ 
 such that $\# f^{-1}(Q')= \deg{f}$ for
all $Q'\in (\IP^g\ssm Z)(\IC)$. 
But $W$, being a non-empty open subset of $\IP^g(\IC)$, 
 is Zariski dense in $\IP^g$ and hence
must meet $(\IP^g\ssm
Z)(\IC)$.
We obtain $d \le\deg{f}$.
\end{proof}

The previous lemma can fail with $\IP^g$  replaced by a
(non-normal) variety. Indeed, the normalization morphism of a
 curve with a node has degree $1$ but more than one point above the node.

Let $\O{1}$ denote the unique ample generator of the Picard group of
projective space. 
For an irreducible closed subvariety $X$ of projective space we let $\deg{(X)}$ be its
geometric degree $(\O{1}^{\cdot \dim X}.[X])$;
we refer to Chapters 1 and 2 \cite{Fulton} for a treatment of the
 intersection theory
needed here. 

We come to a preliminary height lower bound which depends on a Theorem
of Siu.

\begin{lemma}
\label{lem:heightineq}
Let $X\subset \IP^n$ be an irreducible closed subvariety defined over
$\IQbar$ of dimension $g\ge 1$.
 Let $f:X\dashrightarrow \IP^g$ be the
rational map given by $f=[F_0:\cdots:F_g]$ where $F_i$ 
are homogeneous polynomials that are not all identically  zero on $X$ and have
 equal degree at most $D\ge 1$. 
 There is a constant $c=c(X,f)$ 
and a proper
  and Zariski closed subset $Z\subset X$ 
such that $F_0,\ldots,F_g$ have no common zeros on $X\ssm Z$ and
  \begin{equation*}
    \height{f(P)} \ge 
\frac{1}{4^g \deg{(X)}}\frac{\deg{f}}{D^{g-1}} \height{P} - c
  \end{equation*}
for all $P\in (X/Z)(\IQbar)$. 
\end{lemma}
\begin{proof}
Without loss of generality we may assume $\deg{f} \ge 1$. 

The rational map $f$ need not be a morphism of
  varieties. In order to resolve the points of indeterminacy 
we define the Zariski closure of its graph
  \begin{equation*}
    \Gamma = \overline {\{(P,f(P));\,\, P\in \dom{f}(\IQbar)\}} \subset
    X\times\IP^g. 
  \end{equation*}
This is an irreducible projective variety with  $\dim \Gamma = \dim X
= g$. 
The morphism $\dom{f}\hookrightarrow \Gamma$ determined by 
$P\mapsto (P,f(P))$ is birational.
Say $\pi_1:\Gamma\rightarrow X$ and $\pi_2:\Gamma\rightarrow \IP^g$
are the two projection morphisms. Then $\pi_2(\Gamma) = \IP^g$ by
the Fiber Dimension Theorem and since $\pi_2|_\Gamma$ has 
finite fibers generically.

From functorial  properties of the height 
 we  see that  
\begin{equation*}
  \heightlb{P,P'}{\Gamma,\pi_1^*\O{1}|_X} = \height{P}
\quad\text{and}\quad
  \heightlb{P,P'}{\Gamma,\pi_2^*\O{1}} = \height{P'}
\quad\text{for}\quad (P,P')\in \Gamma(\IQbar)
\end{equation*}
 are valid choices for height functions; we recall that
$\heightS$ is the projective height.
In order to prove this lemma it is enough to show that there exists
$c\in\IR$ with
\begin{equation*}
  \heightlb{Q}{\Gamma,\pi_2^*\O{1}} \ge \frac{1}{4^g \deg{(X)}} \frac{\deg{f}}{D^{g-1}}
\heightlb{Q}{\Gamma,\pi_1^*\O{1}|_X} - c,
\end{equation*}
for all $Q$ in a Zariski open dense subset of $\Gamma$. Using 
 functorial  properties of the height, this
inequality holds if  some positive integral power of  the line bundle
\begin{equation*}
  \pi_2^*\O{1}^{\otimes 4^g \deg{(X)} D^{g-1}} \otimes \pi_1^*\O{1}|_X^{\otimes(-\deg{f})}
\end{equation*}
admits a non-zero global section.
By a result of Siu, Theorem 2.2.15 \cite{PosAlgGeom}, such a
  section exists provided we have the
following inequality on intersection numbers
\begin{equation*}
  \left((\pi_2^*\O{1}^{\otimes 4^g \deg{(X)}D^{g-1} })^{\cdot g}.[\Gamma]\right)
\stackrel{?}{>} g \left(\pi_1^*\O{1}|_X^{\otimes\deg{f}}.
(\pi_2^*\O{1}^{\otimes 4^g \deg{(X)}D^{g-1}})^{\cdot(g-1)}.[\Gamma]\right).
\end{equation*}
By multilinearity of intersection numbers  the left-hand side is
$4^{g^2}\deg{(X)}^gD^{g(g-1)}(\pi_2^*\O{1}^{\cdot g}.[\Gamma])$
while the right-hand side is
$4^{g(g-1)}g\deg{(X)}^{g-1} D^{(g-1)^2}(\deg{f})( \pi_1^*\O{1}|_X.\pi_2^*\O{1}^{\cdot
  (g-1)}.[\Gamma])$. 
Hence our lemma follows if we can prove
\begin{equation}
\label{eq:wts}
  4^g \deg{(X)} D^{g-1} \left(\pi_2^*\O{1}^{\cdot g}.[\Gamma]\right) \stackrel{?}{>} g(\deg{f})
  \left(\pi_1^*\O{1}|_X.\pi_2^*\O{1}^{\cdot(g-1)}.[\Gamma]\right). 
\end{equation}

We proceed by proving this inequality. 
The projection  formula implies 
\begin{equation*}
(\pi_2^*\O{1}^{\cdot g}.[\Gamma]) 
= (\deg{\pi_2})(\O{1}^{\cdot g}.[\pi_2(\Gamma)]) 
 = (\deg{\pi_2}) (\O{1}^{\cdot g}.[\IP^g]) = \deg{\pi_2}.  
\end{equation*}
The birational morphism $\dom{f}\hookrightarrow \Gamma$ composed with
$\pi_2$ is nothing other then $f:\dom{f}\rightarrow \IP^g$. Hence
 we have
$\deg{\pi_2}=\deg{f}$ and so 
\begin{equation*}
 (\pi_2^*\O{1}^{\cdot g}.[\Gamma]) = \deg{f}. 
\end{equation*}
By (\ref{eq:wts}) it suffices to show
\begin{equation}
 \label{eq:step1}
  4^g \deg{(X)} D^{g-1} \stackrel{?}{>}
g\left(\pi_1^*\O{1}|_X.\pi_2^*\O{1}^{\cdot(g-1)}.[\Gamma]\right).
\end{equation}

Let $\rho_{1,2}$ denote the projections of $\IP^n\times\IP^g$ onto the
first and second factor, respectively. 
If $Z\subset\IP^n\times\IP^g$ is an irreducible closed subvariety we set
\begin{equation*}
  \hpS{Z} =  \sum_{\atopx{i+j = \dim Z}{i,j\ge 0}} 
{\dim Z \choose i} \left(\rho_1^*\O{1}^{\cdot i}.\rho_2^*\O{1}^{\cdot j}.[Z]\right) U^i V^j
\in \IZ[U,V].
\end{equation*}
This is  the highest homogeneous part of the biprojective
Hilbert polynomial of $Z$ multiplied by $(\dim Z)!$,
cf. \cite{Philippon}. It is
 homogeneous of degree $\dim Z$
 with non-negative integer coefficients. In particular,
 $\hpS{Z}(D,1)\ge 0$. 

Our projective variety $\Gamma$ is an irreducible component of the
intersection of $X\times\IP^g$ with the set of common zeros
 of 
 \begin{equation}
\label{eq:Gammapolys}
 F_i(X_0,\ldots,X_n) - Y_i
\in \IQbar[X_0,\ldots,X_n,Y_0,\ldots,Y_g]\quad (0\le i
\le g);  
 \end{equation}
here $X_i, Y_i$ are projective coordinates on $\IP^n$ and
$\IP^g$, respectively.
These polynomials are  bihomogeneous of bidegree $(\deg F_i,1)$. 
We recall $\deg F_i \le D$. 
Philippon's Proposition
3.3 \cite{Philippon} implies
$\sum_{\Gamma'} \hp{D,1}{\Gamma'} \le\hp{D,1}{X\times \IP^g} $
where the sum runs over all irreducible components $\Gamma'$ 
cut out on $X\times\IP^g$ by the polynomials (\ref{eq:Gammapolys}). 
For any $\Gamma'$ we have
  $\hp{D,1}{\Gamma'} \ge 0$. By forgetting about all irreducible
components except $\Gamma$ we see
\begin{equation}
\label{eq:philipponthm}
 \hp{D,1}{\Gamma} \le \hp{D,1}{X\times \IP^g}.
\end{equation}

Now  ${g \choose
  1}\left(\pi_1^*\O{1}|_X.\pi_2^*\O{1}^{\cdot(g-1)}.[\Gamma]\right)D$ is 
one term in the sum $\hp{D,1}{\Gamma}$. Since all other terms are
non-negative, (\ref{eq:philipponthm}) gives
\begin{equation}
\label{eq:inbound}
  g\left(\pi_1^*\O{1}|_X.\pi_2^*\O{1}^{\cdot(g-1)}.[\Gamma]\right)D \le
  \hp{D,1}{X\times\IP^g}. 
\end{equation}

To complete the proof of (\ref{eq:step1})
 we now bound $\hp{D,1}{X\times\IP^g}$ from above. 
We have $\dim X\times\IP^g = 2g$, so by definition
\begin{equation}
\label{eq:hpD1}
  \hp{D,1}{X\times\IP^g} = \sum^{2g}_{i = 0} 
{2g \choose i} 
\left(\rho_1^*\O{1}^{\cdot
  i}.\rho_2^*\O{1}^{\cdot (2g-i)}.[X\times\IP^g]\right) D^i. 
\end{equation}

Two applications of the projection formula lead to
\begin{alignat}1
\label{eq:projformula}
 (\rho_1^*\O{1}^{\cdot
  i}.\rho_2^*\O{1}^{\cdot (2g-i)}.[X\times\IP^g])
&=  \left(\O{1}^{\cdot
  i}.{\rho_1}_*(\rho_2^*\O{1}^{\cdot (2g-i)}.[X\times\IP^g])\right) \\
\nonumber
& = \left(\O{1}^{\cdot
  (2g-i)}.{\rho_2}_*(\rho_1^*\O{1}^{\cdot i}.[X\times\IP^g])\right).
\end{alignat}
The cycle class $\rho_2^*\O{1}^{\cdot
  (2g-i)}.[X\times\IP^g]$
on $\IP^n\times\IP^g$ is trivial if $2g-i > g$
and $\rho_1^*\O{1}^{\cdot i}.[X\times\IP^g]$ is trivial if $i >\dim X
=g$. 
Therefore, all terms in (\ref{eq:hpD1}) with $i\not=g$ vanish. We are left with 
\begin{alignat*}1 
  \hp{D,1}{X\times\IP^g} &=
{2g \choose g} 
\left(\rho_1^*\O{1}^{\cdot g}.
\rho_2^*\O{1}^{\cdot g}.[X\times\IP^g]\right) D^g.
\end{alignat*}

We find 
$(\rho_1^*\O{1}^{\cdot
  g}.\rho_2^*\O{1}^{\cdot g}.[X\times \IP^g]) = \deg{(X)}$
on inserting $i=g$ in (\ref{eq:projformula}).
We recall (\ref{eq:inbound}) and conclude
\begin{equation*}
(\pi_1^*\O{1}|_X.\pi_2^*\O{1}^{\cdot(g-1)}.[\Gamma]) \le \frac 1g {2g \choose
    g} \deg{(X)}D^{g-1}. 
\end{equation*}

So inequality
(\ref{eq:step1}) holds true since $ {2g \choose g} < 4^g$. As
stated above, this completes the proof. 
\end{proof}

Before we come to the proof of Proposition \ref{prop:heightlb} we 
give an explicit formula for the duplication morphism on $\EL$. 

\begin{lemma}
\label{lem:duplication}
  Let $N\in\IN$, there exist polynomials
$G_{N,0},G_{N,1},G_{N,2}\in \IZ[X_0,X_1,X_2,X_3]$
of total degree at most $2\cdot 4^N$ and homogeneous of degree
 $4^N$
in $X_0,X_1,X_2$
 with the following properties. 
If $([x:y:z],\lambda)\in\EL(\IC)$ then 
$G_{N,i}(x,y,z,\lambda)\not=0$ for some $i\in\{0,1,2\}$ and
\begin{equation*}
  [2^N]([x:y:z],\lambda) = 
([G_{N,0}(x,y,z,\lambda):G_{N,1}(x,y,z,\lambda):G_{N,2}(x,y,z,\lambda)],\lambda).
\end{equation*}
\end{lemma}
\begin{proof}
Let 
\begin{alignat*}1
  G_{1,0} &= 2X_1 X_2 ^3 X_3^2 + (2X_0^3X_1 - 6X_0^2X_1 X_2)X_3 + (2X_0^3X_1 + 2X_0X_1^3), \\
  G_{1,1} &=(-4X_0^2X_2^2 + 6X_0X_2^3 - X_2^4)X_3^3 + (-X_0^4 + 9X_0^3X_2 - 17X_0^2X_2^2 +
  6X_0 X_2^3 - 4X_1^2X_2^2)X_3^2 \\ 
& \quad+ (-2X_0^4 + 9X_0^3X_2 - 4X_0^2X_2^2 + 3X_0X_1^2X_2 -
  4X_1^2X_2^2)X_3 + (-X_0^4 + X_1^4), \\
G_{1,2} &= 8X_1^3X_2.
\end{alignat*}
These three polynomials have no common zeros on $\EL\subset\IP^2\times\B$.
They are homogeneous of degree $4$ in $X_0,X_1,X_2$ and of degree at most $3$
in $X_3$.
If $([x:y:z],\lambda)\in\EL(\IC)$, 
 the duplication formula on page 59 \cite{Silverman:AEC} implies
\begin{equation*}
  [2]([x:y:z],\lambda) = 
([G_{1,0}(x,y,z,\lambda):G_{1,1}(x,y,z,\lambda):G_{1,2}(x,y,z,\lambda)],\lambda).
\end{equation*}

We define $G_{Ni} =
G_{1,i}(G_{N-1,0},G_{N-1,1},G_{N-1,2},X_3)$ inductively.
These polynomials describe $[2^N]$ since
 $G_{N,0},G_{N,1},G_{N,2}\in\IZ[X_0,X_1,X_2,X_3]$ have no
common zero on $\EL$.
By induction we find that $G_{Ni}$
are homogeneous of degree $4^N$ in $X_0,X_1,X_2$ and 
of degree at most $4^N-1$ in $X_3$.
So their total degree
 is at most $2\cdot 4^N-1$.
\end{proof}

\begin{proof}[Proof of Proposition \ref{prop:heightlb}]
Let $X\subset\AL$ be as in the hypothesis. Recall that 
$\AL\subset (\IP^2)^g\times\IP^1$ is quasi-projective. 
The Segre embedding $(\IP^2)^g\times\IP^1
\hookrightarrow \IP^n$, with $n=2\cdot 3^g-1$,
 enables us to embed $\AL$ into projective space. 
Under this embedding, $\AL$ becomes Zariski open in its Zariski
closure. 
 By abuse of notation we will suppose $\AL\subset\IP^n$. 
If $P \in\AL(\IQbar)$, then by Proposition 2.4.4 \cite{BG} the total height
given by (\ref{eq:deftotalheight}) satisfies
\begin{equation}
\label{eq:heightequality}
  \heightlb{P}{\AL} = \height{P}
\end{equation}
where  the height on the right-hand side is the projective height
of $P\in \IP^n(\IQbar)$. 
Let $X_0,\ldots,X_n$ denote
the projective coordinates on $\IP^n$. 
 Throughout this proof $c_1,c_2,\ldots$ denote
positive constants which are independent of $N$ if not stated otherwise.


Let $C$ be the  curve from Proposition
\ref{prop:counting}.
Any point $Q\in C$ is a non-singular point of $C$ and of
$\AL$. By Example II 8.22.1 \cite{Hartshorne}
 there is a Zariski open neighborhood $V$ of $Q$ in
$\AL$ and  homogeneous polynomials
 $H_1,\ldots,H_g\in \IQbar[X_0,\ldots,X_n]$ 
with $\deg H_1 = \cdots = \deg H_g \le c_1$ such that
\begin{equation*}
  C\cap V
\text{ is cut out on $V$ by $H_1,\ldots,H_g$}.
\end{equation*}
We fix $H_0 \in\IQbar[X_0,\dots,X_n]$ such that $H_0(Q)\not =0$ and
$\deg H_0 = \deg H_1$. 
We replace $V$ by a possibly smaller neighborhood of $Q$ on which
$H_0$ does not vanish.
By quasi-compactness, $C$ can be  covered by $c_2$ such $V$ and   $c_1$ is
 independent of $V$. 

Let $c_3=c > 0$ be from Proposition \ref{prop:counting} and $N$ an integer
with $2^N\ge c_3^{-1}$. 
The proposition gives us
 at least $c_3 (2^N)^{2g} =c_3 4^{gN}$ distinct points $P\in
X(\IC)$ such that 
\begin{equation}
\label{eq:irreduciblecomp}
  \left\{[2^N](P)\right\}\quad\text{is an irreducible component of}\quad
[2^N](X)\cap C.
\end{equation}

By the Pigeonhole Principle and after replacing $c_3$ by $c_3/
c_2$
 we may assume that all  $P$ as above
satisfy  $[2^N](P)\in C\cap V$,
 where $V$ is among  the  fixed Zariski open  sets from the
 covering above. 
After replacing $c_3$ by $c_3/(n+1)$, we may suppose that some fixed coordinate of all  $P$ is non-zero.
These $P$ are then contained in a non-empty Zariski open  subset of
$\AL$
on which the Segre morphism can be inverted using monomials.

We use Lemma \ref{lem:duplication} to see that $[2^N]$
equals $[G_0:\cdots:G_n]$ on a Zariski open and non-empty subset of
$\AL$
where $G_i\in\IZ[X_0,\ldots,X_n]$ have suitably bounded degree.
Let $H_0,\ldots,H_g$ be the polynomials attached to $V$. We set
$F_i = H_i(G_0,\ldots,G_n)\in\IQbar[X_0,\ldots,X_n]$ for $0\le i \le g$. The
$F_i$ are homogeneous with $\deg{F_i}\le c_4 4^N$. 

The rational map $f:X\dashrightarrow \IP^g$ given 
by $f = [F_0:\cdots:F_g]$ is regular at the points $P$ considered
above; indeed, by construction $[G_0(P):\cdots:G_n(P)] = [2^N](P)$ lies in
$V(\IC)$ and is thus not a zero of $H_0$.
 Moreover, we have $f(P) = [1:0:\cdots:0]$.
We claim that each $P$ is an irreducible component of
$f^{-1}([1:0:\cdots:0])$. 
We aim for a contradiction by assuming that 
 there is an irreducible curve
$Y\subset\dom{f}$ containing $P$ with $f(Y) = [1:0:\cdots:0]$. Without
loss of generality we may assume $Y\subset [2^N]^{-1}(V)$. But then
$[2^N](Y) \subset V$ is in the set of common zeros of $H_1,\ldots, H_g$,
hence
$[2^N](Y) \subset [2^N](X)\cap C$. Now  $[2^N](Y)$
remains a curve and it contains $[2^N](P)$. But this contradicts 
(\ref{eq:irreduciblecomp}).

So the fiber $f^{-1}([1:0:\cdots:0])$ contains at least $c_3 4^{gN}$
isolated points. By Lemma \ref{lem:degree} we conclude $\deg{f}\ge c_3
4^{Ng}$. 

The proposition will now follow  from Lemma
\ref{lem:heightineq} applied to the Zariski closure of $X$ in
$\IP^n$. Indeed, taking $D= c_4 4^N$ we get 
$\height{f(P)} \ge c_5 4^N \height{P} - c_6(N)$
for all $P\in U(\IQbar)$ where $U\subset X$ is Zariski open and
dense and $c_6(N)$ is a constant which may depend on $N$. 
From (\ref{eq:heightequality}) we conclude 
$\height{f(P)} \ge c_5 4^N \heightlb{P}{\AL} - c_6(N)$.
After shrinking $U$ we may assume that $U\subset V$.  
 Hence $f(P) = H([2^N](P))$ with 
$H = [H_0:\cdots:H_g]$ and thus
\begin{equation}
\label{eq:heightflb}
\height{H([2^N](P))} \ge c_5 4^N \heightlb{P}{\AL} - c_6(N).
\end{equation}

Using the local definition of the projective height given in 
Chapter 1.5 \cite{BG}
together with the
triangle, respectively the ultrametric inequality we find $c_7
\ge 0$ such that
\begin{equation*}
  \height{H(P')} \le c_7 \max\{1,\height{P'}\}
\end{equation*}
for all $P'\in\IP^n(\IQbar)$ such that 
at least one $H_i(P')\not=0$.
If $P'=[2^N](P)$ we obtain
 $\height{H([2^N](P))} \le c_7\max\{1,\height{[2^N](P)}\}$. 
But $\height{[2^N](P)} = 
\heightlb{[2^N](P)}{\AL}$
 by (\ref{eq:heightequality}).
The proposition now follows
from 
(\ref{eq:heightflb}). 
\end{proof}

\section{Passing to the N\'eron-Tate Height and Proof of the Main Results}
\label{sec:mainresults}

\subsection{A Weak Version of  Theorem \ref{thm:main} for the Legendre Family}

The height involved in the upper bound of
 Proposition \ref{prop:heightlb} is the total  height given by
(\ref{eq:deftotalheight}). On a fixed
 fiber of $\AL$ above $\B(\IQbar)$ this height differs from
 the N\'eron-Tate height  (\ref{eq:defntheight}) by a bounded function. 
The dependency of this bound on the fiber was made explicit by
Silverman-Tate \cite{Silverman} and Zimmer \cite{Zimmer}. See also the
related result of Zarhin and Manin \cite{ZarhinManin}.

 \begin{theorem}
\label{thm:stz}
There is an absolute constant $c>0$ such that if
$P\in\AL(\IQbar)$, then 
\begin{equation*}
  |\heightlb{P}{\AL}-\ntheightlb{P}{\AL}|\le
  c\max\{1,\height{\piL(P)}\}. 
\end{equation*}
 \end{theorem}
 \begin{proof}
This follows from either Zimmer's Theorem  or from
 the result
of Silverman and Tate, cf. Theorem A \cite{Silverman}.
 \end{proof}

On $X$ we can now bound $\height{\piL(P)}$ from above in terms
of $\ntheightlb{P}{\AL}$ by applying Proposition \ref{prop:heightlb}
to a 
 sufficiently large but fixed  integer $N$. 

\begin{lemma}
\label{lem:passNTheight}
  Let $X\subset\AL$ be an irreducible closed subvariety
  defined over $\IQbar$ which dominates $\B$, has dimension $g$, and 
satisfies (i) or (ii) of Proposition \ref{prop:counting}. There exist a
constant $c = c(X) > 0$ and
  a non-empty Zariski open subset $U\subset X$ such that 
  \begin{equation*}
    \height{\piL(P)} \le c \max\{1,\ntheightlb{P}{\AL}\}
  \end{equation*}
 for all $P\in U(\IQbar)$. 
\end{lemma}
\begin{proof}
  Let $c_1 > 0 $ be the constant $c$ in Proposition \ref{prop:heightlb}
  and let $c_2$ be $c$ from Theorem \ref{thm:stz}.
  We fix $N$ to be the least  positive integer with $2^N\ge
  \max\{2^{1/c_1},\sqrt{2c_2/c_1}\}$. 
Proposition \ref{prop:heightlb} implies
\begin{equation*}
 c_1 4^N \heightlb{P}{\AL}-c_3(N)\le \heightlb{[2^N](P)}{\AL}
\end{equation*}
for all $P\in U(\IQbar)$ where  $c_3(N)$ is $c'(N)$ from said proposition.
We use the bound for
$|\heightlb{P}{\AL}-\ntheightlb{P}{\AL}|$
to obtain the second inequality in
\begin{equation*}
 c_1 4^N \height{\pi_L(P)}-c_3(N)\le 
 c_1 4^N \heightlb{P}{\AL}-c_3(N)\le 
\ntheightlb{[2^N](P)}{\AL}+c_2\max\{1,\height{\piL(P)}\},
\end{equation*}
the first one follows from (\ref{eq:deftotalheight}).

The N\'eron-Tate height is quadratic. Dividing by $4^N$ leads to
\begin{equation*}
\left( c_1-\frac{c_2}{4^N}\right) \height{\piL(P)} \le
 \ntheightlb{P}{\AL}+\frac{c_3(N)+c_2}{4^N}.
\end{equation*}
By our choice of $N$ we have $c_1-c_2/4^N\ge c_1/2 > 0$ and the current
lemma follows.
\end{proof}

The previous lemma only holds for 
 hypersurfaces in $\AL$. Moreover, the
restriction (i) in  Proposition \ref{prop:counting}
is stronger than $\remtor{X}\not=\emptyset$, the
implicit condition of Theorem \ref{thm:main}(ii).
We address both issues in the next lemma which is 
essentially an induction on dimension.

\begin{lemma}
\label{lem:heightubU}
  Let $X\subset\AL$ be an irreducible closed  subvariety 
defined over $\IQbar$
 which is not an irreducible component of a flat subgroup scheme
  of $\AL$. 
There is a constant $c = c(X)> 0$ and
  a Zariski open non-empty set $U\subset X$ such that 
  \begin{equation}
\label{eq:heightubU}
 \height{\piL(P)} \le c\max\{1,\ntheightlb{P}{\AL}\}
  \end{equation}
 for all $P\in U(\IQbar)$. 
\end{lemma}
\begin{proof}
The hypothesis implies $X\not=\AL$.
The lemma certainly holds for all $X$ which do not dominate $\B$ because
$\piL|_X$ is then constant.
Hence we reduce to the case where $X$
dominates $\B$.
  The proof is by induction on $\dim X + \dim\AL = \dim X + g + 1$. 
 The induction parameter is at
  least $3$. 

If $\dim X +\dim\AL  =3$, then the only possibility
is $\dim X = 1$ and $\dim\AL=2$, so 
$\AL=\EL$. 
By hypothesis, we are in case (ii) of Proposition \ref{prop:counting}. 
The height inequality (\ref{eq:heightubU}) follows from
 Lemma \ref{lem:passNTheight}.

So let us assume $\dim X +\dim\AL  \ge 4$. 
Because $\dim X <\dim  \AL$ we split into
two  cases.

The first case is when $X$ is no hypersurface, so $\dim X\le g-1$.
 We pick any
 $\lambda\in \B(\IQbar)$. All irreducible components of the fiber
 $X_\lambda$ have dimension at least $\dim X -1$ by the 
Fiber Dimension Theorem. But since $X$ dominates $\B$, all these irreducible
 components have dimension precisely $\dim X-1$. 
 Moreover,  any $X_\lambda$ is non-empty because $\piL|_X:X\rightarrow \B$ is
dominant and proper, hence surjective. Let $Z\subset X_\lambda$ be an
 irreducible component. 
 The whole fiber
 $(\AL)_\lambda = (\EL)_\lambda^g$ is a power of an
 elliptic curve. So there is a projection $\Psi:
 (\EL)_\lambda^g\rightarrow (\EL)_\lambda^{\dim X-1}$ 
 onto $\dim X-1$ coordinates such that $\dim Z = \dim \Psi(Z)$. We again
 apply the Fiber Dimension Theorem to find a point $P\in
 Z(\IQbar)$
 with $\dim_P \Psi|_Z^{-1}(\Psi(P)) = 0$. We can even arrange that $P$
  is not contained in any other irreducible component of $X_\lambda$. 
 Of course, $\Psi$ extends to a projection
 $\AL\rightarrow\EL\times_\B\cdots\times _\B\EL$
 ($\dim X-1$ factors). 
 An irreducible component
 of
 $\Psi|_X^{-1}(\Psi(P))\subset X_\lambda$  containing $P$ must be
 $\{P\}$.
 A final
 application of the Fiber Dimension Theorem shows
 $\dim\Psi(X) \ge \dim X$. But the reverse inequality  also
 holds. So $\dim \Psi(X) = \dim X$.
To  simplify notation we assume that $\Psi$
 projects onto the first $\dim X -1$ coordinates. For $\dim X\le j\le
 g$ let
 $\Psi_j:\AL\rightarrow\EL\times_\B\cdots\times_\B\EL=\mathcal{B}$
 ($\dim X$ factors)
be the projection onto the first $\dim X-1$ and the
 $j$-th coordinate. We note $\dim \mathcal{B} = 1+\dim X$.
Then $\Psi_j$ is proper and so $\Psi_j(X)$ is an irreducible  closed
subvariety  of $\mathcal{B}$.
It follows quickly that $\dim \Psi_j(X) = \dim X$, so $\Psi_j(X)$ has
codimension $1$. 
We claim that at least one $\Psi_j(X)$ is not contained in a proper
flat subgroup scheme of $\mathcal{B}$. 
Indeed, otherwise $X$ would be contained in a flat
 subgroup scheme of dimension $\dim X$ by
Lemma \ref{lem:horizontal}(ii) and (iii).  Hence  $X$  would be an
irreducible component of a flat subgroup scheme. 
This is impossible by
hypothesis. So let us assume that  $X'=\Psi_j(X)$ is not contained in a proper flat subgroup scheme of
$\mathcal{B}$. Since $ X' \not= \mathcal{B}$ we conclude that $X'$ 
 is not an irreducible component of a flat  subgroup scheme.
Because $\dim X'+\dim\mathcal{B} = 2\dim X + 1 \le \dim X +
\dim\AL-1$ we may apply induction.
So there is $c_1 > 0$ and a non-empty Zariski
open subset $U'\subset X'$ such that 
$\height{\piL(Q)} \le c_1 ( 1+\ntheightlb{Q}{\mathcal{B}})$ 
for all $Q\in U'(\IQbar)$.
The current case  follows with $U=\Psi_j|_X^{-1}(U')$ because 
$\piL(P)=\piL(\Psi_j(P))$ and 
$\ntheightlb{\Psi_j(P)}{\mathcal{B}}\le
\ntheightlb{P}{\AL}$; in fact, we are just omitting certain coordinates.

The second   case is when $X$ is a hypersurface, so $\dim X = g$; it is here where we apply
Lemma \ref{lem:passNTheight}. We  split up into two
subcases. In the first subcase we suppose $\dim \varphi(X) \ge 2$
for all $\varphi\in \IZ^g\ssm\{0\}$. That is, $X$
satisfies  hypothesis (i) of Proposition \ref{prop:counting}. We
conclude this subcase immediately
by applying Lemma \ref{lem:passNTheight}. It remains to treat the case
where there exists $\varphi\in \IZ^g\ssm\{0\}$
such that $\dim \varphi(X) \le 1$, in this case $C=\varphi(X)$ 
is an irreducible closed curve in $\EL$. 
It satisfies property (ii) of Proposition \ref{prop:counting}.
We note that $4\le \dim X
+\dim\AL$ hence
$\dim C + \dim \EL < \dim X + \dim\AL$. By induction
there is $c_2 > 0$
with $\heightlb{Q}{\EL}\le
c_2\max\{1,\ntheightlb{Q}{\EL}\}$ 
for all $Q\in
C(\IQbar)$; indeed any non-empty Zariski open  subset of $C$ misses
merely finitely many points of $C$. Writing $Q=\varphi(P)$ for some $P\in
X(\IQbar)$ we see
$\height{\piL(P)} = \height{\piL(\varphi(P))}\le c_2
\max\{1,\ntheightlb{\varphi(P)}{\EL}\}$ 
for all $P\in
X(\IQbar)$.

 It is well-known how to bound $\ntheightlb{\varphi(P)}{\EL}$ from
above in terms of $\ntheightlb{P}{\AL}$.
Indeed, say  $\varphi=(a_1,\ldots,a_g)$. If $P_1,\ldots,P_g\in(\EL)_\lambda(\IQbar)$ where
$\lambda\in\B(\IQbar)$, then 
\begin{alignat*}1
 \ntheightlb{[a_1](P_1)+\cdots + [a_g](P_g)}{\EL}&\le
g(a_1^2\ntheightlb{P_1}{\EL}+\cdots
+a_g^2\ntheightlb{P_g}{\EL})\\ &\le g\max\{a_1^2,\ldots,a_g^2\}\ntheightlb{P}{\AL}; 
\end{alignat*}
this follows from the fact that the N\'eron-Tate height is a quadratic
form and from the Cauchy-Schwarz inequality.
So
$\ntheightlb{\varphi(P)}{\EL}\le g
\max\{a_1^2,\ldots,a_g^2\}\ntheightlb{P}{\AL}$.
We conclude $\height{\piL(P)} \le c_3\max\{1,\ntheightlb{P}{\AL}\}$
for all $P\in X(\IQbar)$ where $c_3$ is independent of $P$. 
\end{proof}

\subsection{Adding Level Structure and Proof of Theorem \ref{thm:main}}

Let $\mathcal{E},\mathcal{A},S,$ and $\pi$ be as in the introduction.
We recall that the curve  $S$ is
defined over $\IQbar$.
We fix an irreducible and non-singular projective curve $\overline S$
and assume $S$ is Zariski open in $\overline{S}$. 
 Let $\mathcal{L}$ be
a line bundle on $\overline S$ and let $\heightlbS{\overline
  S,\mathcal{L}}$ be
a choice of height function  ${\overline S}(\IQbar)\rightarrow\IR$.

Before coming to the proof of Theorem \ref{thm:main}(ii) we need an
auxiliary construction.

\begin{lemma}
\label{lem:levelstructure}
  Let us assume that $\mathcal{A}$ is not isotrivial. After possibly
  replacing $S$ by a non-empty Zariski open subset 
there exists an irreducible non-singular quasi-projective curve $S'$
  defined over $\IQbar$ with the following property. We have a
  commutative diagram
\begin{equation}
\label{eq:leveldiagramm}
  \bfig 
   \hSquares/{<-}`>`>`>`>`{<-}`>/
             [\mathcal{A}`\mathcal{A}'`\AL`S`S'`\B; f ` e `\pi ` ` \piL`l`\lambda]
  \efig
\end{equation}
where $l$ is finite, $\lambda$ is quasi-finite, 
$\mathcal{A}'$ is the abelian scheme 
$\mathcal{A}\times_S S'$,
 $f$ is finite and flat, and $e$ is quasi-finite and
flat. 
Moreover, the restriction of $f$ and $e$ to any fiber
of $\mathcal{A}'\rightarrow S'$ is an isomorphism of abelian varieties.
Finally, if $P\in \mathcal{A}'_s(\IQbar)$, then
\begin{equation}
  \label{eq:NTequality}
  \ntheightlb{f(P)}{\mathcal{A}} = \ntheightlb{e(P)}{\AL}.
\end{equation}
\end{lemma}
\begin{proof}
Let $j:S\rightarrow Y(1)$ be the morphism
as before Lemma \ref{lem:jnonconst}. It is non-constant by said lemma.

We  regard $j$ as an element of $\IQbar(\overline S)$, the
function field of $\overline S$. 
We may fix a finite field extension $K$ of 
 $\IQbar(S)$ 
such that 
the generic fiber of $\mathcal{E}\rightarrow
S$ is isomorphism, over $K$, to an elliptic curve determined by
$y^2=x(x-1)(x-\lambda)$ with $\lambda\in K$.
Then $K$ is the function field of an irreducible non-singular projective curve $\overline
S'$ and the inclusion $\IQbar(S)\subset K$ induces 
 a finite  morphism   $l:\overline S'\rightarrow \overline S$.
The rational function $\lambda\in \IQbar(\overline S')$ extends to a morphism
$\lambda:\overline{S}'\rightarrow\IP^1$.

We let $S'$ denote the preimage of $S$ in $\overline{S'}$.
Hence we obtain a 
irreducible non-singular quasi-projective  curve $S'$ over $\IQbar$ 
such that
\begin{equation*}
\bfig
\square[S'`\B`S`Y(1);\lambda`l`
\lambda\mapsto 2^8\frac{(\lambda^2-\lambda+1)^3}{\lambda^2(\lambda-1)^2}
`j]
\efig
\end{equation*}
commutes; by abuse of notation we use the symbols $\lambda$ 
and $l$ to denote their restrictions to $S'$.
The restricted morphism $l:S'\rightarrow S$ remains finite.
 We see that $\lambda$ is
a non-constant morphism between two irreducible curves; hence it is quasi-finite.
After replacing $S$ and $S'$ by Zariski open subsets we may assume
that $\lambda:S'\rightarrow \B$ is \'etale and $l:S'\rightarrow S$ is
 finite and \'etale. 
After shrinking $S$, we still have an abelian scheme
$\mathcal{A}\rightarrow S$.

The fibered product $\mathcal{E}' = \mathcal{E}\times_S S'$ is an
abelian scheme over $S'$ with  elliptic curves as fibers. 
By construction, the generic fiber of $\mathcal{E}'\rightarrow S'$ 
is isomorphic,  over $\IQbar(S')$, to the elliptic curve
defined by $y^2=x(x-1)(x-\lambda)$.
 By Proposition 8, page 15 \cite{NeronModels}   an  isomorphism on the
 generic fiber extends to an isomorphism on the whole abelian scheme.
We remark that no argument in this paper relies on the existence of a
N\'eron model. 
 This
 provides the top arrow in  the commutative square on the left in
 \begin{equation*}
 \bfig
  \hSquares/>`>`>`>`>`=`>/[\mathcal{E}' ` S'\times_\B \mathcal{E}_L ` \mathcal{E}_L `
    S' ` S' ` \B;``````\lambda]
 \efig   
 \end{equation*}
the square on the right is Cartesian. 


Let $\mathcal{A}'$ be the $g$-fold fibered power of $\mathcal{E}'$
over $S'$. We remark that there is a natural morphism
$\mathcal{A'}\rightarrow \B$.
We take the product over $\B$ of the 
morphisms $\mathcal{A}'\rightarrow \mathcal{E}'\rightarrow \EL$ coming
from the $g$ projections 
to get the right square in
(\ref{eq:leveldiagramm}),
the square on the left is the product over $S$ of the $g$ morphisms
$\mathcal{A}'\rightarrow \mathcal{E}'\rightarrow \mathcal{E}$.

We claim that $e$ is flat. By Corollary 11.3.11 \citeEGAIVIII it suffices to
prove that the following statement. Say $x$ is in
$\B$ below $s$, a point of $S'$. We must show 
that $e$ restricts to a flat morphism
$\mathcal{A}'_x\rightarrow (\AL)_x$ where
$\mathcal{A}'_x$ and  $(\AL)_x$ are the fibers of
$\mathcal{A}'\rightarrow \B$ and $\AL\rightarrow \B$ above $x$, respectively.
We consider the scheme theoretic fiber $\lambda^{-1}(x) \rightarrow \spec{k(x)}$.
Since $S'\rightarrow \B$ is \'etale, $\lambda^{-1}(x)$ is \'etale over
$\spec{k(x)}$.
We have a natural morphism $\spec{k(s)}\rightarrow \lambda^{-1}(x)$
which when composed
with $\lambda^{-1}(x)\rightarrow \spec{k(x)}$ is the  \'etale morphism
$\spec{k(s)}\rightarrow \spec{k(x)}$. So $\spec{k(s)}\rightarrow
\lambda^{-1}(x)$ is \'etale, cf. Corollary 17.3.5 \citeEGAIVIV, and in
particular flat. 
The induced morphism
$\mathcal{A}'_x\times_{\lambda^{-1}(x)} \spec{k(s)}\rightarrow
\mathcal{A}'_x$ is flat too; but this new fibered product is
 $\mathcal{A}'_s$. Now the composition $\mathcal{A}'_s\rightarrow
\mathcal{A}'_x\rightarrow (\AL)_x$ is an isomorphism of
abelian varieties over $k(x)$ and therefore  flat. 
Corollary 2.2.11(iv) \citeEGAIVII 
  implies that
$\mathcal{A}'_x\rightarrow (\AL)_x$ is flat at all points
in the image of $\mathcal{A}'_s$. If we let $s$ run over all points in
the fiber of $S'\rightarrow \B$ above $x$ we conclude that
$\mathcal{A}'_x\rightarrow (\AL)_x$ is flat. In a similar
way one can show that $f$ is flat. 

Since $l$ and $\mathcal{A}\rightarrow S$ are proper, we see that
$\mathcal{A}'\stackrel{f}{\rightarrow}
\mathcal{A}\stackrel{\pi}{\rightarrow} S$ is proper
by (\ref{eq:leveldiagramm}).
Therefore, $f$ is proper. 
We have $\dim \mathcal{A} = \dim \mathcal{A}'  =\dim \AL = g + 1$
and Corollary III 9.6 \cite{Hartshorne} implies
that $f$ and $e$ are quasi-finite. So $f$ is finite. 

By construction, the restriction of $f$ and $e$ to a fiber of
$\mathcal{A}'\rightarrow S'$ determines an isomorphism of abelian
varieties. 

For each $s'\in S'(\IQbar)$  we have
  isomorphisms $\mathcal{E}'_{s'}\rightarrow \mathcal{E}_{l(s')}$ and 
$\mathcal{E}'_{s'}\rightarrow (\EL)_{\lambda(s')}$
of elliptic curves.
Hence 
$\mathcal{E}_{l(s')}$ and $(\EL)_{\lambda(s')}$ are isomorphic.
By our definition made in the introduction
we see that N\'eron-Tate height of a point in $\mathcal{E}_{l(s')}$
equals the N\'eron-Tate height of its image in $(\EL)_{\lambda(s')}$.
Passing to the product gives (\ref{eq:NTequality}).
\end{proof}


\begin{lemma}
\label{lem:heightubUS}
Let us assume that $\mathcal{A}$ is not isotrivial.
  Let $X\subset\mathcal{A}$ be an irreducible closed  subvariety 
defined over $\IQbar$
 which is not an irreducible component of a flat subgroup scheme
  of $\mathcal{A}$. 
There is a constant $c = c(X)> 0$ and
  a non-empty Zariski open  subset $U\subset X$ such that 
  \begin{equation}
\label{eq:heightubUS}
    \heightlb{\pi(P)}{\overline S,\mathcal{L}} \le c\max\{1,\ntheightlb{P}{\mathcal{A}}\}
  \end{equation}
 for all $P\in U(\IQbar)$. 
\end{lemma}
\begin{proof}
We  keep the notation from the previous lemma. 
Without loss of generality we may replace $S$ by the non-empty Zariski
open subset given there. 
Using
 Corollary III 9.6 \cite{Hartshorne} and the fact that $f$ is finite and flat we
 see that
the irreducible components of
$f^{-1}(X)$ have dimension $\dim X$.
As $f$ is an open and closed morphism, it is surjective. So
$f^{-1}(X)$ is non-empty. We  pick an irreducible component
 $X''$ of $f^{-1}(X)$ and define $X'$ to be the Zariski closure of
$e(X'')$ in $\AL$. Since $e|_{X''}$ has finite fibers, the
Fiber Dimension Theorem implies $\dim X' = \dim X'' = \dim X$.

 We claim that $X'$
is not an irreducible component of a subgroup scheme of $\AL$. 
We assume the contrary and deduce a contradiction.
  Let $G\subset \AL$ be a subgroup scheme whose
  irreducible components dominate $\B$ and such that one of them is $X'$.
Then  $G$ is equidimensional of dimension $\dim X$ by Lemma \ref{lem:horizontal}(i).
Since $e$ is flat and quasi-finite we may again conclude that
 $e^{-1}(G)$ is equidimensional of dimension $\dim G=\dim X$.
We see that $X''$ is an irreducible component of $e^{-1}(X')$.
The fact that $f$ is finite  implies that $f(e^{-1}(G))$
is equidimensional of dimension $\dim X$ and  $\dim f(X'') = \dim
X$. 
Since $X$ is irreducible and because $f$ is closed, we have
$f(X'')=X$. We remark $f(X'')\subset f(e^{-1}(G))$, 
so $X$ is an irreducible component of $f(e^{-1}(G))$.
Latter is a subgroup scheme of $\mathcal{A}$ by Lemma
\ref{lem:levelstructure}, 
hence it remains to
show that all irreducible components of $f(e^{-1}(G))$ dominate $S$.
An irreducible component of $e^{-1}(G)$ has dimension $\dim G$.
Its image under the quasi-finite morphism $e$ 
is Zariski dense in some irreducible component of $G$.
 Therefore, any irreducible component of $e^{-1}(G)$
dominates $S'$. Its image under $f$ dominates $S$ since $l:S'\rightarrow
S$ is dominant.
From this we conclude that any irreducible component of $f(e^{-1}(G))$
dominates $S$. We have a contradiction.

Thus we may apply Lemma \ref{lem:heightubU} to $X'$ and obtain a
Zariski open and non-empty $U'\subset X'$ on which the height
inequality holds.
 Certainly, $e|_{X''}^{-1}(U')$ is Zariski open in $X''$ and non-empty. From above we
 have $f(X'')=X$, so
 $f(e|_{X''}^{-1}(U'))$ contains a non-empty Zariski open subset $U$ of $X$.

We claim that (\ref{eq:heightubUS}) holds on $U(\IQbar)$.
Let $P\in U(\IQbar)$ lie above $s\in S(\IQbar)$ and say
$P'' \in e^{-1}|_{X''}(U')(\IQbar)$ with $f(P'')=P$ and $e(P'')=P'\in U'(\IQbar)$.
If $P''$ and $P'$ lie above $s''\in S'(\IQbar)$ and $s' \in
\B(\IQbar)$, respectively, then  chasing around
(\ref{eq:leveldiagramm})
yields $\lambda(s'') = s'$ and $l(s'')=s$.

We know that $\height{s'} \le c_1 \max\{1,\ntheightlb{P'}{\AL}\}$
for some constant $c_1>0$ which does not depend on $P$.
We recall that $\heightS$ is the projective height on $\IP^1(\IQbar)$.
By height properties, $\heightS\circ \lambda$ is a choice
for a representative of 
 $\heightlbS{\overline{S}',\lambda^*\O{1}}$; with this choice we have
$\heightlb{s''}{\overline{S}',\lambda^*\O{1}} = \height{s'}$.
Now $\lambda$ is finite and $\O{1}$ is ample, so $\lambda^*\O{1}$ is
ample. Therefore, there is a positive integer $a$ such that
$\lambda^*\O{1}^{\otimes a}\otimes l^*\mathcal{L}^{\otimes(-1)}$
is ample.
Functorial properties of the height imply that
$a\heightlbS{\overline {S}',\lambda^*\O{1}} \ge \heightlbS{\overline
  {S},\mathcal{L}}\circ l  - c_2$ on $\overline{S}'(\IQbar)$ for some
constant $c_2$.
On inserting $s''$ we find
$\heightlb{s}{\overline{S},\mathcal{L}}\le
c_3\max\{1,\ntheightlb{P'}{\AL}\}$ for some constant $c_3>0$ which is
independent of $P$.
Finally, (\ref{eq:NTequality}) implies
$\ntheightlb{P'}{\AL} = \ntheightlb{P}{\mathcal{A}}$ and this
completes the proof.
\end{proof}


\begin{proof}[Proof of Theorem \ref{thm:main}(ii)]
  We prove the height inequality in the assertion by induction on the dimension. The case of
  dimension $0$ being trivial we assume $\dim X \ge 1$. 
  If $X$ is an irreducible component of a flat subgroup scheme, then
  $\remtor{X}=\emptyset$ and there is nothing to prove.
So we may assume the
  contrary. By Lemma \ref{lem:heightubUS} inequality
  (\ref{eq:mainthmineq}) holds  on $(X\ssm Z)(\IQbar)$ for some
proper  Zariski closed subset
  $Z\subsetneq X$. Let $Z=Z_1\cup\cdots\cup Z_r$ be
  the decomposition into irreducible components.
It suffices to show  (\ref{eq:mainthmineq})  on
all $Z_i(\IQbar)\cap \remtor{X}(\IQbar)$.
Since $\dim Z_i \le \dim X-1$ we may do  induction on the
dimension. 
 We obtain the desired inequality
for all $P$ in  $(X\ssm Z)(\IQbar) \cup \remtor{Z_1}(\IQbar)\cup\cdots\cup
\remtor{Z_r}(\IQbar)$.  This set contains 
$\remtor{X}(\IQbar)$ by  a formal argument using the definition of
$\remtor{X}$.
\end{proof}

The next lemma implies part (i) of Theorem \ref{thm:main}. 

\begin{lemma}
  Let $X\subset\mathcal{A}$ be an irreducible closed subvariety
  defined over $\IC$. 
  Then $\remtor{X}$ is Zariski
  open in $X$ and empty if and only if $X$ is itself an irreducible
  component of a flat subgroup scheme of $\mathcal{A}$. 
\end{lemma}
\begin{proof}
Without loss of generality we may assume that $X$ dominates $S$,
otherwise $\remtor{X}=X$ by definition.

To prove the lemma 
 it suffices to show the following statement. There are at most
finitely many irreducible subvarieties of $X$ that are
irreducible
 components of a flat  subgroup scheme of
$\mathcal{A}$ and  maximal with this property. 

Let $Z$ be such a subvariety; it must dominate $S$.
The generic fibers $Z_\eta$ and $X_\eta$ of $\pi|_Z$ and $\pi|_X$,
respectively,  are 
subvarieties of $\mathcal{A}_\eta$, the generic fiber of $\pi$.
Let $Y_\eta\subset
X_\eta$ be a further variety 
which is defined and irreducible over
$\IC(S)$ and with $Z_\eta\subset Y_\eta$. 
We also assume that $Y_\eta$ is an irreducible component
of an algebraic subgroup of $\mathcal{A}_\eta$. Therefore, it is an irreducible
component
of the kernel $\ker\Psi_\eta$ for an endomorphism
$\Psi_\eta$ of $\mathcal{A}_\eta$. 
By Proposition 8, page 15 \cite{NeronModels} we may extend $\Psi_\eta$
to an endomorphism $\Psi$ of $\mathcal{A}$.
Let $Y$ be the Zariski closure of $Y_\eta$ in $\mathcal{A}$.
Then $Y_\eta$ is the generic fiber of $\pi|_Y$ by 
 Proposition 2.8.5 \citeEGAIVII and the comment after its proof.
 Then $Y\subset\ker\Psi$ and $Z\subset Y \subset X$. By comparing  dimensions using
the Fiber Dimension Theorem
 one shows that $Y$ is an irreducible component of
$\ker\Psi$. Therefore, $Z=Y$ by maximality. 
So,  $Z_\eta = Y_\eta$. 
We have just shown that $Z_\eta$ is an  irreducible subvariety of
$X_\eta$ 
which is an irreducible component of an algebraic subgroup of
$\mathcal{A}_\eta$
and which is maximal with this property. By Raynaud's Theorem, the Manin-Mumford
Conjecture, $Z_\eta$ comes from a finite set of subvarieties of
$X_\eta$. 
But $Z$ is the Zariski closure of $Z_\eta$ in $\mathcal{A}$ and we see
that there are only finitely many such $Z$.

It follows that $X\ssm \remtor{X}$ is a finite
union of irreducible components of flat subgroups schemes of
$\mathcal{A}$.
The second claim of the lemma follows too.
\end{proof}

\subsection{Special Points on $\mathcal{A}$ and Proof of Theorem \ref{thm:special1}}

We recall that the $j$-invariant of an elliptic curve with complex
multiplication is an algebraic number. Therefore, it makes sense to
speak of its Weil height.

The following result of Poonen is needed for the proof of 
Theorem \ref{thm:special1}.

\begin{lemma}[Poonen]
\label{lem:cmheight}
Let $T\in\IR$. 
 Up-to $\IQbar$-isomorphism there are only finitely many elliptic
  curves over $\IQbar$ with complex multiplication and
 whose $j$-invariant has Weil   height at most $T$.
\end{lemma}
\begin{proof}
  This is Lemma 3 \cite{Poonen:MRL01}.
\end{proof}

\begin{proof}[Proof of Theorem \ref{thm:special1}] 
We recall that $\mathcal{A}$ and $S$ are defined over $\IQbar$.
By hypothesis $\mathcal{A}$ is not isotrivial. By Lemma \ref{lem:jnonconst}
the morphism $j:S\rightarrow Y(1)$ which associates to $s\in S(\IC)$
the $j$-invariant of $\mathcal{E}_s$ is dominant.

We begin with the elementary ``if'' direction.
Say there is $s\in S(\IC)$ such that $X$ 
is an irreducible component of an algebraic subgroup of 
$\mathcal{A}_s$ and such that $\mathcal{A}_s$  has complex multiplication. 
The claim follows since the set of torsion points of an abelian
variety lies Zariski dense.
 Now say
$X$ is an irreducible component of a flat subgroup scheme of
$\mathcal{A}$.
The set of $s \in S(\IQbar)$ such that
$\mathcal{A}_s$ has complex multiplication is infinite, hence
Zariski dense in $S$. For any such $s$, the fiber $X_s$ is an
algebraic subgroup of $\mathcal{A}_s$ and thus contains a
Zariski dense set of torsion points.  It follows that if $Z$ is
 the Zariski closure of all special points in $X$, then $Z\cap X_s =
 X_s$ for infinitely many $s\in S(\IQbar)$. Therefore,
 $Z=X$ by a dimension argument. 

We now prove the ``only if'' direction. 
Let us assume that $X$ is an irreducible closed subvariety of
$\mathcal{A}$ which contains a Zariski dense set of special
points. Let us also assume that $X$ is not an irreducible component 
of a flat subgroup scheme of $\mathcal{A}$. 

If $s\in S(\IC)$ such that $\mathcal{A}_s$ has complex multiplication then so does
$(\EL)_s$ 
and it follows that
$j(s)$ is algebraic. But since $S$ is defined over $\IQbar$ we see $s\in S(\IQbar)$.
Moreover, any  $P\in \mathcal{A}(\IC)$ which is a torsion point of
$\mathcal{A}_{\pi(P)}$ and for which this fiber has complex
multiplication must be algebraic. 
Hence $X$ contains a Zariski dense set of algebraic points.
It follows that $X$ is defined over
$\IQbar$. 

Let $\overline S$ be as before the statement of
Theorem \ref{thm:main}
and $\mathcal{L}$ and ample
line bundle on $\overline S$.
By part (i) of said theorem
we see that $ \remtor{X}$ is  non-empty and Zariski open in $X$.
So there is Zariski dense subset of points in 
$\remtor{X}(\IQbar)$ which are torsion in a fiber with
complex multiplication. Let $P$ be in this set and $s=\pi(P)$.
The N\'eron-Tate height of $P$ vanishes because this point is torsion.
So
$\heightlb{s}{\overline S,\mathcal{L}}$ is bounded from above 
independently of $P$ by  Theorem \ref{thm:main}.

The curve $\overline S$ is projective and non-singular  so $j$ extends to a
morphism
$\overline S\rightarrow \IP^1$.
By properties of the height,  
the projective height on $\IP^1(\IQbar)$ is a valid choice for the
representative of 
 $\heightlbS{\IP^1,\O{1}}$ and 
$\heightS\circ j$ is a valid representative for 
 $\heightlbS{\overline S,j^*\O{1}}$.
Since $\mathcal{L}$ is ample, there is a positive integer $a$ such
that
$\mathcal{L}^{\otimes a}\otimes j^*\O{1}^{\otimes(-1)}$ is ample. 
Functorial properties of the height imply that 
$a \heightlbS{\overline S,\mathcal{L}} - \heightlbS{\overline
  S,j^*\O{1}}$
is bounded from below on $\overline S(\IQbar)$.
It follows that  $\height{j(s)}$ is bounded from above
independently of $P$.
  Poonen's result implies 
 that the set of possible $j$-invariants of
$\mathcal{E}_{s}$ is finite. So there
 are only finitely many possible $s$.
Hence the Zariski dense set of $P$ we consider is  in  
 finitely many fibers of $\pi|_X:X\rightarrow S$.
This means that  $X$ is contained in 
$\mathcal{A}_s$ for some $s\in S(\IQbar)$ and
this fiber must have complex multiplication. 
We now regard $X$ as a subvariety of  the fixed
abelian variety $\mathcal{A}_s$. By hypothesis $X$ contains a
Zariski dense set of torsion points. The classical Manin-Mumford
conjecture implies that $X$ is an irreducible component of an
algebraic subgroup of $\mathcal{A}_s$. So $X$ is as in
case (i) of the definition of  special subvarieties given before
Theorem \ref{thm:special1}. 
\end{proof}

\subsection{Proof of Theorem \ref{thm:specialE}}

Theorem \ref{thm:specialE} is a  consequence of Theorem \ref{thm:main}
 and a result of Szpiro and Ullmo which is encapsulated in
 the next lemma.

Any elliptic curve $E$ over $\IQbar$ has a semi-stable Faltings height
$\Fh{E}$
which depends only on its $\IQbar$-isomorphism class, see
\S 1 \cite{Silverman:HEC}.

\begin{lemma}
\label{lem:heckeorbitheight}
Let $E$ be an elliptic curve defined over $\IQbar$ without complex
multiplication.
  There exists a constant $c=c(E)$  with the following
  property. If $E'$ is an elliptic curve such that there exists a
   isogeny $E\rightarrow E'$ with cyclic kernel of cardinality  $N$, then 
  \begin{equation*}
    \Fh{E'} \ge \Fh{E} + \frac 12 \log (N) - c\log\log(3N). 
  \end{equation*}
\end{lemma}
\begin{proof}
Let $E,E',$ and $N$ be as in the hypothesis.
For a prime $p$ let $e_p$ be the exponent of $p$ in the factorization
of $N$.
Szpiro and Ullmo's Th\'eor\`eme 1.1 \cite{SU:variation} implies
\begin{equation}
\label{eq:SU}
  \Fh{E'} \ge \Fh{E} + \frac 12 \log N - \sum_{p|N}
  \frac{p^{e_p}-1}{(p^2-1)p^{e_p-1}} \log p - c_1
\end{equation}
where the constant $c_1$  may depend  $E$, but not on $E'$ or $N$. We have
\begin{equation*}
  \frac{p^{e_p}-1}{(p^2-1)p^{e_p-1}}   \le 
\frac{p^{e_p}}{p^2p^{e_p-1}/2} = 
 \frac{2}{p}.
\end{equation*}
And by an elementary calculation
\begin{equation*}
 \sum_{p|N} \frac{p^{e_p}-1}{(p^2-1)p^{e_p-1}}  \log p 
\le 2 \sum_{p|N}\frac{\log p}{p} \le c_2 \log\log(3N),
\end{equation*}
with $c_2$ absolute.
The lemma follows from (\ref{eq:SU}). 
\end{proof}

\begin{lemma} 
\label{lem:heckeheight}
Let $E$ be an elliptic curve defined over $\IQbar$ without complex
multiplication and let $T\in\IR$. 
\begin{enumerate}
\item [(i)] Up-to $\IQbar$-isomorphism there are only finitely many elliptic
  curves  over $\IQbar$ which are isogenous to $E$ and which have Faltings height at most
  $T$.
\item[(ii)] Up-to $\IQbar$-isomorphism there are only finitely many elliptic
  curves over $\IQbar$ which are isogenous to $E$ and whose $j$-invariant has Weil
  height at most $T$
\end{enumerate}
\end{lemma}
\begin{proof}
  The second part of the lemma follows from the first since bounding the
  $j$-invariant of an elliptic curve amounts to bounding its Faltings
  height 
by Proposition 2.1 \cite{Silverman:HEC}.

Let $E'$ be an elliptic curve which is isogenous
to $E$ with $\Fh{E'}\le T$.  There exists an  isogeny $E\rightarrow
E'$ with cyclic kernel of cardinality $N$, say; for a proof we refer to Lemma 6.2
\cite{MW:ellisogenies}. By Lemma \ref{lem:heckeorbitheight} we see that $N$ is
bounded  in terms of $T$. So there are only finitely many
possibilities for the kernel of $E\rightarrow
E'$. Hence up-to $\IQbar$-isomorphism there are only finitely many
possibilities for $E'$.  
\end{proof}

\begin{proof}[Proof of Theorem \ref{thm:specialE}]
The proof runs along the lines of the proof of Theorem
\ref{thm:special1}. If $E$ does not have complex multiplication we
use Lemma \ref{lem:heckeheight} instead of  Lemma \ref{lem:cmheight}.
\end{proof}

\section{An Instance of the Bogomolov Conjecture over Functions
  Fields}
\label{sec:bogo}

Let $K,\overline K,E$ be as in the hypothesis of Theorem
\ref{thm:bogo}.
So $K$ is the function field of an irreducible 
non-singular projective curve $\overline
S$ defined over $\IQbar$. 
By considering a Weierstrass model over $K$ of 
$E$ we see that
 $E$ is the generic fiber of an abelian 
scheme $\mathcal{E}\rightarrow S$; here $S$ is a sufficiently
small  Zariski open and dense
subset of $\overline S$.
As in the introduction, let $\mathcal{A}$ be the $g$-fold 
fibered power of 
$\mathcal{E}$ over $S$ with $\pi:\mathcal{A}\rightarrow S$ the
structural morphism.
The condition that $E$ has non-constant $j$-invariant implies that
$\mathcal{A}$ is not isotrivial. 
We write $A$ for  the generic fiber of $\pi:\mathcal{A}\rightarrow S$;
this is just the abelian variety $E^g$ over $K$.

On $\overline S$ we fix an ample line bundle $\mathcal{L}$. We also choose
a representative of the equivalence class of height functions
 associated to the pair $\overline S,\mathcal{L}$ and denote
 it by $\heightlbS{\overline S,\mathcal{L}}$,
cf. Section \ref{sec:heights}.

Any finite field extension $K'$ of $K$ is the
function field of an irreducible non-singular projective curve $\overline{S}'$ defined over
$\IQbar$. The inclusion $K\subset K'$ induces a finite morphism $\rho:
\overline{S}'\rightarrow \overline{S}$.
We set $S' = \rho^{-1}(S)$ and regard $K'$ as the function field of $S'$.

A point $x\in A(K')$ induces a rational map $\widetilde x :
S'\dashrightarrow \mathcal{A}$ such that 
$\pi \circ \widetilde x = \rho$ on the domain of $\widetilde x$.

Recall that we defined the N\'eron-Tate height
$\ntheightlbS{\mathcal{A}}$ on  $\mathcal{A}$
in Section \ref{sec:heights}.
For any algebraic point $t\in S'(\IQbar)$ in the domain of
 $\widetilde x$  it makes sense to
speak of $\ntheightlb{\widetilde x(t)}{\mathcal{A}}$.

On $E$ we have a symmetric and ample line bundle coming from the zero
element of $E$ considered as a Weil divisor.
Taking the tensor product of the pull-backs coming from the $g$
projections 
 $E^g\rightarrow E$ determines
 a symmetric and ample line bundle on $A$. Since
$K$ is equipped with
 a product formula in the sense of Chapter 1.4 \cite{BG}, 
we may associate to said line bundle 
a N\'eron-Tate height $\ntheightlbS{A}$; cf. Chapter 9.2 of the same reference.

 Of course,  $\ntheightlbS{A}$
 need not equal the height appearing in Theorem \ref{thm:bogo}.
However, functorial properties of the N\'eron-Tate height imply the
following statement. If $\ntheightlbS{A}'$ is a  N\'eron-Tate height on
$A(\overline K)$
coming from a  symmetric and ample line bundle there exists $c
> 0$ such that $\ntheightlbS{A}\le c \ntheightlbS{A}'$.
Therefore, it suffices to prove Theorem \ref{thm:bogo} with the fixed
height function described above.


We can now state Silverman's  Theorem \cite{Silverman} 
applied to our situation.

\begin{theorem}
\label{thm:silverman}
In the notation above,  
let $t_1,t_2,\ldots\in S'(\IQbar)$ be a sequence of points in
the domain of $\widetilde x$
 such that
$\lim_{k\rightarrow \infty}\heightlb{\rho(t_k)}{\overline S,\mathcal{L}}=\infty$.
Then
\begin{equation*}
  \lim_{k\rightarrow\infty}
\frac{\ntheightlb{\widetilde x(t_k)}{\mathcal{A}}}
     {\heightlb{\rho(t_k)}{\overline S,\mathcal{L}}} = 
\ntheightlb{x}{A}.
\end{equation*}
\end{theorem}
\begin{proof}
We apply Silverman's Theorem B to $A_{K'}$, the base change of $A$ to $K'$.
From this we will see that the limit equality holds. Indeed, we have
$\ntheightlbS{A_{K'}} =
[K':K]\ntheightlbS{A}$ on $A_{K'}(K')$. Moreover,
Silverman's choice of height on $\overline{S'}(\IQbar)$ is asymptotically equal to 
$[K':K]^{-1} \heightS\circ\rho$ by functorial properties of the height.
\end{proof}

We now combine the upper bound from  Theorem \ref{thm:main} with the
conclusion of Theorem \ref{thm:silverman} which serves as a competing
lower bound.

\begin{proof}[Proof of Theorem \ref{thm:bogo}]
We let $i:A\rightarrow \mathcal{A}$ be the natural morphism
and let $\mathcal{X}$ be the Zariski closure of $i(X)$ in
$\mathcal{A}$.
Then $\mathcal{X}$ is irreducible. By 
 Proposition 2.8.5 \citeEGAIVII,
 $\mathcal{X}$ is flat over $S$ and 
satisfies $i^{-1}(\mathcal{X}) = X$. Hence $X$ is the generic fiber of 
$\pi|_\mathcal{X}:\mathcal{X}\rightarrow S$, so the Fiber Dimension Theorem implies 
 $\dim X = \dim\mathcal{X} - 1$.

We claim that $\mathcal{X}$ is not an irreducible component of a flat subgroup
scheme of $\mathcal{A}$. Let us assume  the converse, we will arrive
at a contradiction.
A repeated application of Lemma \ref{lem:horizontal}(ii) together
appropriate projections to $\mathcal{E}^{g'}$ with $g'\le g$
gives us independent $\varphi_1,\ldots,\varphi_{g-\dim \mathcal{X} +1}\in\IZ^g$
such that $\mathcal{X}\subset\ker 
(\varphi_1\times_S\cdots\times_S \varphi_{g-\dim\mathcal{X} + 1})$.
This implies a similar inclusion on the generic fiber. The
 common kernel of $\varphi_1,\ldots,\varphi_{g-\dim \mathcal{X} + 1}$
considered as homomorphisms   $A=E^g\rightarrow E$ is an algebraic subgroup of dimension
$\dim \mathcal{X}-1=\dim X$. So it contains $X$ as an irreducible
component, this is a contradiction.


From Theorem \ref{thm:main}(i) we conclude that $\mathcal{X}\ssm 
\remtor{\mathcal{X}}$ is Zariski closed and proper in $\mathcal{X}$. 
Let $c > 0$ be as in part (ii) of this theorem. We claim that 
$c^{-1}$ is a suitable choice for $\epsilon$ and that the preimage
of  $\mathcal{X}\ssm \remtor{\mathcal{X}}$
under the dominant morphism $i|_X : X\rightarrow \mathcal{X}$
  is a suitable choice for $Z$. 
Indeed, let $x\in  (X\ssm Z)(K')$ where $K'$ is a finite field
extension of $K$ contained in $\overline K$.  As above,
 there is an irreducible non-singular projective curve $S'$
over $\IQbar$
with function field $K'$, a finite morphism   $\rho:S'\rightarrow S$,
and a rational map $\tilde x : S'\dashrightarrow \mathcal{A}$ such that 
$\rho = \pi \circ\tilde x$ on the domain of $\tilde x$.

The Zariski closure $\mathcal{Y}$ of the image of $\tilde x$  in $\mathcal{A}$
 is an irreducible subvariety of $\mathcal{A}$. It has
dimension at most $1$. But it must be a curve since $\rho$ is
dominant. We have
$\mathcal{Y}\subset \mathcal{X}$. Finally, 
 $\mathcal{Y} \cap
\remtor{\mathcal{X}}\not=\emptyset$. 
because
$x\notin Z(\overline K)$. 

So $\mathcal{Y} \cap
\remtor{\mathcal{X}}$ is a quasi-projective curve which dominates
$S$. Since $\mathcal{L}$ is ample
there is a sequence of points 
$P_1,P_2,\ldots\in (\mathcal{Y} \cap \remtor{\mathcal{X}})(\IQbar)$ 
 such that 
\begin{equation*} 
 \lim _{k\rightarrow\infty} \heightlb{\pi(P_k)}{\overline S,\mathcal{L}} = \infty.
\end{equation*}
For $k$ large enough there is $t_k\in S(\IQbar)$ with $\widetilde x(t_k) = P_k$. 
Then $\lim_{k\rightarrow\infty}
\heightlb{\rho(t_k)}{\overline{S},\mathcal{L}} = \infty$ because $\pi(P_k) = \rho(t_k)$. 

 Theorem \ref{thm:main}(ii) implies
 \begin{equation*}
   \ntheightlb{\tilde x(t_k)}{\mathcal{A}} = 
   \ntheightlb{P_k}{\mathcal{A}} \ge
c^{-1} \heightlb{\pi(P_k)}{\overline S,\mathcal L}
=\epsilon \heightlb{\rho(t_k)}{\overline S,\mathcal L}
 \end{equation*}
for  $k$ large enough since the left-hand side will eventually be
greater than $1$. So
\begin{equation*}
  \liminf_{k\rightarrow\infty}
\frac{\ntheightlb{\tilde
    x(t_k)}{\mathcal{A}}}{\heightlb{\rho(t_k)}{\overline S,\mathcal L}} \ge \epsilon.
\end{equation*}

By Silverman's Theorem this limes inferior is in fact a limes which
equals $\ntheightlb{x}{A}$. We conclude
$\ntheightlb{x}{A}\ge \epsilon$. The theorem follows
because $\epsilon > 0$ was independent of $x$. 
\end{proof}

\bibliographystyle{amsplain}
\bibliography{literature}

\address{
\noindent
Philipp Habegger,
Johann Wolfgang Goethe-Universit\"at,
Robert-Mayer-Str. 6-8,
60325 Frankfurt am Main,
Germany,
{\tt habegger@math.uni-frankfurt.de}
} 


\noindent MSC 2010: 11G15 (main), 11G18, 11G50, 14G40, 14K22.

\end{document}